\documentclass[reqno, oneside, 12pt]{amsart}

\usepackage[letterpaper]{geometry}
\usepackage{enumerate, hyperref,url,amssymb,amsmath,amsthm,amsxtra,mathtools,mathrsfs,calc,nccmath,color}

\usepackage{calc}
\usepackage{graphicx}
\usepackage{caption}
\usepackage{subcaption}

\usepackage{needspace}

\usepackage{algorithm2e}

\SetStartEndCondition{ }{}{}%
\SetKwProg{Pydef}{def}{\string:}{}
\SetKwComment{tcc}{\#\ }{}
\SetKwFor{For}{for}{\string:}{}%
\SetKwFor{While}{while}{:}{fintq}%
\SetKwIF{If}{ElseIf}{Else}{if}{:}{elif}{else:}{}%
\SetKwFunction{Range}{range}
\SetKw{KwIn}{in}
\SetKw{KwNot}{not}
\SetKw{KwOr}{or}
\SetKw{KwInlineIf}{if}
\AlgoDontDisplayBlockMarkers\SetAlgoNoEnd\SetAlgoNoLine%

\newcommand{\Z}{\mathbb{Z}}
\newcommand{\Q}{\mathbb{Q}}
\newcommand{\R}{\mathbb{R}}
\renewcommand{\H}{\mathbb{H}}

\newcommand{\mQ}{\mathcal{Q}}
\newcommand{\mN}{\mathcal{N}}

\newcommand{\ep}{\varepsilon}

\let\temp\phi
\let\phi\varphi
\let\varphi\temp

\renewcommand{\(}{\left(}
\renewcommand{\)}{\right)}

\newcommand{\ord}{\operatorname{ord}}
\newcommand{\SL}{\operatorname{SL}}
\newcommand{\GL}{\operatorname{GL}}
\newcommand{\Frob}{\operatorname{Frob}}
\newcommand{\Sh}{\operatorname{Sh}}
\newcommand{\Gal}{\operatorname{Gal}}
\newcommand{\Tr}{\operatorname{Tr}}

\newcommand{\sk}{\big|_k }

\newcommand{\pfrac}[2]{\left(\frac{#1}{#2}\right)}
\newcommand{\pmfrac}[2]{\left(\mfrac{#1}{#2}\right)}
\newcommand{\ptfrac}[2]{\left(\tfrac{#1}{#2}\right)}
\newcommand{\pMatrix}[4]{\left(\begin{matrix}#1 & #2 \\ #3 & #4\end{matrix}\right)}
\renewcommand{\pmatrix}[4]{\left(\begin{smallmatrix}#1 & #2 \\ #3 & #4\end{smallmatrix}\right)}

\newcommand{\tx}{\text}
\newcommand{\spmod}[1]{\ensuremath{\,(#1)}}

\newtheorem{theorem}{Theorem}[section]
\newtheorem{lemma}[theorem]{Lemma}
\newtheorem{corollary}[theorem]{Corollary}
\newtheorem{proposition}[theorem]{Proposition}

\theoremstyle{remark}
\newtheorem*{remark}{Remark}

\numberwithin{equation}{section}

\AtBeginDocument{%
   \def\MR#1{}
}

\begin{document}


\title{Scarcity of congruences for the partition function}

\date{\today}
\author{Scott Ahlgren}
\address{Department of Mathematics\\
University of Illinois\\
Urbana, IL 61801} 
\email{sahlgren@illinois.edu} 

\author{Olivia Beckwith}
\address{Department of Mathematics\\
Tulane University\\
New Orleans, LA 70118 }
\email{obeckwith@tulane.edu} 

\author{Martin Raum}
\address{Chalmers tekniska högskola och G\"oteborgs Universitet\\
Institutionen för Matematiska vetenskaper\\
SE-412 96 Göteborg, Sweden}
\email{martin@raum-brothers.eu}

\thanks{The first author was  supported by a grant from the Simons Foundation (\#426145 to Scott Ahlgren).
The third author was partially supported by Vetenskapsr\aa det Grant~2019-03551.
}

 
\begin{abstract}  The arithmetic properties of the ordinary partition function  $p(n)$ have been the topic of intensive study for the past century.
Ramanujan proved that there  are linear congruences of the form $p(\ell n+\beta)\equiv 0\pmod\ell$ for the primes $\ell=5, 7, 11$, and it is known that
there are no others of this form.  On the other hand, for every prime $\ell\geq 5$ there are infinitely many examples of congruences of the form
$p(\ell Q^m n+\beta)\equiv 0\pmod\ell$ where $Q\geq 5$ is prime and $m\geq 3$. This leaves open the question of the existence of such congruences when $m=1$ or $m=2$
(no examples in these cases are known).  We prove in a precise sense that such congruences, if they exist, are exceedingly scarce.  Our methods involve
a careful study of modular forms of half integral weight on the full modular group which are related to the partition function.  Among many other tools, we use 
 work of Radu which describes expansions of such modular forms along square classes at cusps of the modular curve $X(\ell Q)$, Galois representations 
 and the arithmetic large sieve.
\end{abstract}


\maketitle


 \section{Introduction}
The partition function $p(n)$ counts the number of ways to represent the positive integer~$n$ 
as the sum of a non-increasing sequence of positive integers.  By convention we agree that $p(0):=1$ and that $p(n):=0$ if $n\not\in \{0, 1, 2, \dots\}$.
This function has been long-studied in combinatorics and number theory. 
Ramanujan proved the famous congruences 
\begin{gather}\label{eq:ramcong}
p(\ell n+\beta_\ell)\equiv 0\pmod\ell
\quad\text{for }\ell=5, 7, 11
\tx{,}
\end{gather}
where $\beta_\ell:=\frac1{24}\pmod\ell$
(when we speak of such a  congruence we   mean  that the partition values vanish modulo $\ell$ for all integers $n$).
Much of the interest in such congruences arises from the relationship between the partition function and Dedekind's eta function.  These congruences are a prototypical example of arithmetic phenomena which occur for a wide class of weakly holomorphic modular forms.

Ramanujan conjectured extensions of these results for arbitrary powers of these primes.  Proofs in the cases $\ell=5, 7$ are attributed to Ramanujan and Watson 
\cite{ram_cong_mathz, ram_cong_PLMS, ram_someprop, watson}; the 
case $\ell=11$ is much more difficult and was resolved by Atkin \cite{atkin_glasgow}.  
Several decades later,  examples of congruences for primes $\ell\leq 31$ were found by Newman, Atkin,  and O'Brien \cite{newman, atkin-obrien, atkin_mult}.
These examples are not as simple as \eqref{eq:ramcong}; they take the form
\begin{gather}\label{eq:partcong}
p(\ell Q^m n+\beta)\equiv 0\pmod\ell
\tx{,}
\end{gather}
where $Q$ is a prime distinct from $\ell$.
In Ono's groundbreaking work \cite{Ono_annals} it was shown  for each prime $\ell\geq 5$ that there are infinitely many primes $Q$ for which
\eqref{eq:partcong} holds with $m=4$.  This result was subsequently generalized in several directions  in the case of the partition function \cite{Ahlgren_mathann, Ahlgren-Ono}
and for general weakly holomorphic modular forms \cite{Treneer_1, Treneer_2}.   
For a more complete history,  one may consult for example
\cite{berndt-ono, ahlgren-ono_expos, ahlgren-ono_conj}.

While there is a long literature proving that congruences of the form \eqref{eq:partcong} exist, there are  fewer results on the non-existence of congruences. Ramanujan \cite{ram_cong_PLMS} speculated that the  congruences \eqref{eq:ramcong} are the only examples of the form $p(\ell n+\beta)\equiv 0\pmod\ell$.  Kiming and Olsson \cite{Kiming-Olsson} proved that any example must have $\beta\equiv\frac1{24}\pmod\ell$.   Ramanujan's speculation was later confirmed by the first author and Boylan \cite{Ahlgren-Boylan}.
Radu \cite{Radu_subbarao} confirmed an old conjecture of Subbarao by proving that there are no congruences of the form
\begin{gather*}
p(mn+\beta)\equiv 0\pmod\ell
\quad\text{with }\ell=2, 3
\tx{.}
\end{gather*}
In later work \cite{Radu_aoconj}, Radu confirmed a conjecture of the first author and Ono by proving that if there is a congruence
\begin{gather*}
p(mn+\beta)\equiv 0\pmod\ell
\quad\text{with }\ell\geq 5\ \ \text{prime}
\tx{,}
\end{gather*}
then $\ell\mid m$ and $\pfrac{1-24\beta}\ell\in \{0, -1\}$; these results have been generalized to a wide class of weakly holomorphic modular forms
and mock theta functions 
 \cite{ahlgren-kim, andersen}.

After Ono's work, we know that for any prime $\ell\geq 5$ there are infinitely many primes $Q$ for which there is a congruence of the form
\begin{gather*}
p\(\ell Q^4 n+\beta\)\equiv 0\pmod \ell.
\end{gather*}
Atkin \cite{atkin_mult} discovered congruences of the form
\begin{gather}\label{eq:atkincong}
p\(\ell Q^3 n+\beta\)\equiv 0\pmod \ell.
\end{gather}
For $\ell=5, 7$ and $13$, he proved that such congruences exist for  infinitely many $Q$;
for example, there are $30$ values of $\beta$ which give rise to a congruence
\begin{gather*}
p\(13\cdot11^3\, n+\beta\)\equiv 0\pmod{13}.
\end{gather*}
Atkin describes a method to produce such congruences which requires ``accidental" values of certain Hecke eigenvalues,
and he gives examples of congruences \eqref{eq:atkincong} for $\ell\leq 31$.   Using this method, 
 Weaver \cite{weaver} and Johansson \cite{johansson} found many more examples (now more than $22$ billion) for these primes $\ell$.
 In recent work \cite{ahlgren-allen-tang} the first author, Allen and Tang have shown that there are infinitely many congruences of the form \eqref{eq:atkincong}
for every prime $\ell\geq 5$.

On the other hand, we know from \cite{Ahlgren-Boylan} that there are no congruences 
\begin{gather*}
p\(\ell n+\beta\)\equiv 0\pmod\ell
\quad\text{with }\ell\geq 13.
\end{gather*}
In view of these results, it is natural to ask if there are any congruences of the form
\begin{gather}\label{eq:Qellcong}
p(\ell Q n + \beta_{\ell, Q}) \equiv 0 \pmod{\ell}
\end{gather}
or
\begin{gather}\label{eq:quadcong}
p(\ell Q^2n+\beta_{\ell, Q})\equiv 0\pmod\ell.
\end{gather}
For every $Q$  such congruences   occur for $\ell=5, 7$, or $11$ because of the original Ramanujan congruences.
However, no other example of congruences \eqref{eq:Qellcong}, \eqref{eq:quadcong} is known (and we will show below that the first example, if it exists, would
involve very large values of $\ell$ and $Q$).

It is therefore natural to speculate that there are no such congruences outside of the trivial examples arising from \eqref{eq:ramcong}.
The goal of this paper is to prove that such congruences, if they exist,  are extremely scarce.  

To state the first result, suppose that $\ell\geq 5$  and $Q\geq 5$ are primes with $Q\neq \ell$. 
As mentioned above,  Radu \cite{Radu_aoconj} proved that \eqref{eq:Qellcong} can occur only when 
\begin{gather*}
\pmfrac{1-24\beta_{\ell, Q}}\ell\in \{0, -1\}.
\end{gather*}

We begin by stating one  of our main results in terms of the partition function (a stronger version is given below in Theorem~\ref{thm:main}).
  The theorem shows that either congruences modulo $\ell$ are scarce, or that      values of $p(n)$
 which are not divisible by $\ell$ are 
themselves  scarce (as described in \eqref{eq:pscarce}).    It is most useful to consider this in the context of  \eqref{eq:fldeldef}, 
which relates the partition values in \eqref{eq:pscarce} to the coefficients of a modular form of half-integral weight.  In particular, we would have \eqref{eq:pscarce} if the form $f_{\ell, \delta}$ which is defined in~\eqref{eq:fldelprop} below were congruent modulo $\ell$ to an iterated derivative of a theta series; this is explained in more detail in  the remark following Theorem~\ref{thm:main}.
 \begin{theorem} \label{thm:mainpart}
Suppose that $\ell\geq 5$ is prime, and
fix $\delta\in \{0, -1\}$.
Let $S$ be the set of primes $Q$ for which there exists a congruence of the form \eqref{eq:Qellcong}  with $\pfrac{1-24\beta_{\ell, Q}}\ell=\delta$.
Then one of the following is true.
\begin{enumerate}
\item $S$ has density zero, or
\item 
we have 
 \begin{gather}\label{eq:pscarce}
\#\left\{ n\leq X\ : \pmfrac{-n}\ell=\delta,  \ \ p\pmfrac{ n+1}{24}\not\equiv0\pmod\ell\right\}\ll \sqrt{X} \log X.
\end{gather}
\end{enumerate}
\end{theorem}

Using the    statement \eqref{eq:tq0} below, we are able 
 to obtain the following corollary.
\begin{corollary} \label{cor:nocase2}
 Suppose that $17\leq \ell< 10,000$.    Let $S$ be the set of primes $Q$ for which we have a congruence
\begin{gather*}
p(\ell Q n+\beta_{\ell, Q})\equiv 0\pmod\ell.
\end{gather*}
Then $S$ has density zero.
\end{corollary}
\begin{remark}
In the case~$\ell = 13$, statement~\eqref{eq:tq0}, which we falsify to derive Corollary~\ref{cor:nocase2}, holds for~$\pfrac{1-24\beta_{\ell, Q}}\ell = \delta = -1$ by Atkin's work~\cite{atkin_multiplicative}. When~$\ell = 13$ and~$\delta = 0$, the set $S$ has density zero.
\end{remark}

We are able to show that there are no congruences \eqref{eq:Qellcong} 
for a large range of $\ell$ and $Q$. The method of computation  relies heavily on results of Radu~\cite{Radu_aoconj} and requires computing relatively few  values of the partition function. In particular, we have 
\begin{theorem}\label{thm:nosmallcong}
Apart from the cases  arising from the Ramanujan congruences \eqref{eq:ramcong}, 
there are no congruences~\eqref{eq:Qellcong} for~$\ell<1,000$ and~$Q<10^{13}$, and for~$\ell<10,000$ and~$Q<10^{9}$.
\end{theorem}

 \begin{remark}
If $Q$ is any positive integer, then by Theorem~4.3 and Lemma~4.5 of \cite{Radu_aoconj}, the non-existence of a congruence with modulus $\ell Q$ 
implies the non-existence of congruences with modulus $\ell Q Q'$ for many integers $Q'$.
  \end{remark}
  
Much of the interest in the arithmetic properties of the partition function arises from its connection to modular forms.
The Dedekind eta function is defined for $\tau$ in the upper complex half-plane $\H$  by 
\begin{gather*}
\eta(\tau):=q^\frac1{24}\prod_{n=1}^\infty \(1-q^n\), 
\end{gather*}
with the standard notation $q:=e^{2\pi i \tau}$. 
This is a modular form of weight $\frac12$ on the full modular group (details are given in the next section) which is connected to $p(n)$ via 
 Euler's formula
\begin{gather*}
\frac1{\eta(\tau)}=\sum p\pmfrac {n+1}{24}q^\frac n{24}=q^{-\frac1{24}}+q^\frac{23}{24}+2q^\frac{47}{24}+\dots.
\end{gather*}
 Ramanujan's congruences may be more succinctly written in the form 
\begin{gather*}
\mfrac1{\eta(\tau)}\Big|U_\ell=\sum p\pmfrac{\ell n+1}{24}q^\frac{n}{24}\equiv 0\pmod\ell\ \ \text{for}\ \ \ell=5, 7, 11,
\end{gather*}
where $U_\ell$ is the standard operator defined in Lemma~\ref{lem:uQ}.

Our results are stated most naturally in terms of a family of modular forms.
If $k\in \frac12\Z$, $N$ is a positive integer, and $\nu$ is a multiplier system on $\Gamma_0(N)$ in weight $k$, then we denote
by $S_k\(N, \nu\)$ the space of modular forms of weight $k$ and multiplier $\nu$ on $\Gamma_0(N)$
(details  will be given in the next section).
With this notation, we have
\begin{gather*}
\eta\in S_\frac12\(1, \nu_\eta\),
\end{gather*}
where $\nu_\eta$ is  the multiplier described in \eqref{eq:nueta}.

For $\delta\in \{0, -1\}$, there is a modular form 
\begin{gather}\label{eq:fldelprop}
f_{\ell, \delta}\in 
\begin{cases}
S_{\frac{\ell^2 - 2\ell}{2}} \( 1, \nu_\eta^{-1}\)\ \ &\text{if $\delta=0$}\tx{,}\\
S_{\frac{\ell^2 - 2}{2}} \( 1, \nu_\eta^{-1}\)\ \ &\text{if $\delta=-1$}\tx{,}
\end{cases}
\end{gather}
with 
\begin{gather}\label{eq:fldeldef}
f_{\ell, \delta}=\sum a_{\ell,\delta}(n)q^\frac n{24} 
\equiv\sum_{\pfrac{-n}\ell=\delta}p\pmfrac{n+1}{24}q^\frac n{24}\pmod\ell.
\end{gather}

\begin{remark}
We have defined the modular forms in this manner for ease of exposition throughout the paper.
When  $\delta=0$, the relevant values of the partition function are in fact captured by a modular form of lower weight.  
In particular,  defining  
\begin{gather*}
F_\ell:=f_{\ell, 0}\big|U_\ell\in S_{\frac{\ell - 2}{2}} \( 1, \nu_\eta^{-\ell}\),
\end{gather*}
we have 
\begin{gather*}
F_\ell\equiv\sum p\pmfrac{\ell n+1}{24}q^\frac n{24} \pmod{\ell}.
\end{gather*}

\end{remark}

For each prime $Q\geq 5$ we have a Hecke operator $T_{Q^2}$ (defined in  \eqref{eq:heckedef}) on the space $S_k\(1, \nu_\eta^{-1}\)$.  The following result implies  Theorem~\ref{thm:mainpart} above.
\begin{theorem} \label{thm:main}
Suppose that $\ell\geq 5$ is prime, and
fix $\delta\in \{0, -1\}$.
Let $S$ be the set of primes $Q$ for which there exists a congruence of the form \eqref{eq:Qellcong}  with $\pfrac{1-24\beta_{\ell, Q}}\ell=\delta$.
Then one of the following is true.
\begin{enumerate}
\item $S$ has density zero, or
\item 
we have 
 \begin{gather}\label{eq:al0}
 \#\left\{ n\leq X\ : a_{\ell,\delta} (n) \not\equiv0\pmod\ell\right\}\ll \sqrt{X} \log X
 \end{gather}
 and 
\begin{gather}\label{eq:tq0}
f_{\ell, \delta}\big| T_{Q^2}\equiv 0\pmod\ell\quad\text{for all primes }Q\equiv -1\pmod\ell.
\end{gather}

\end{enumerate}
\end{theorem}

\begin{remark}
By work of Kiming-Olsson \cite{Kiming-Olsson}  we know that $f_{\ell,-1}\not\equiv 0\pmod\ell$, and 
by work of the  first author with Boylan \cite{Ahlgren-Boylan} we know that $f_{\ell,0}\not\equiv 0\pmod\ell$ if $\ell\geq 13$.
Apart from the three cases in which these modular forms vanish$\pmod\ell$, a result of Bella\"iche,  Green and  Soundararajan \cite{soundetal} (which improves previous results
\cite{ahlgren-boylan_odd}, \cite{ahlgren_odd} by a $\log$-factor) implies that we have the lower bound
\begin{gather*}
\#\left\{ n\leq X\ : a_{\ell,\delta} (n) \not\equiv0\pmod\ell\right\}\gg \frac{\sqrt X}{\log\log X}.
\end{gather*}
This is a natural barrier in  this setting due to the presence of modular forms which are congruent to theta functions and their derivatives.
\end{remark}

\begin{remark} 
In terms of the partition function,  
condition \eqref{eq:tq0} is equivalent to the statement that for all $n$ and $Q$ with $\pfrac{-n}\ell=\delta$ and $Q\equiv -1\pmod\ell$, we have
\begin{gather}\label{eq:partitionhecke}
p\pmfrac{Q^2n+1}{24}+Q^{-2}\pmfrac{-12 n}Qp\pmfrac{n+1}{24}+Q^{-3}p\pmfrac{\frac n{Q^2}+1}{24}\equiv 0\pmod\ell
\end{gather}
(this can be seen from a computation involving the weights of the modular forms $f_{\ell, \delta}$).
\end{remark}

The next main result gives a strong necessary condition for the existence of a congruence~\eqref{eq:Qellcong}.
It is a consequence of more general results obtained for squarefree modulus $Q$ in Section~\ref{sec:squarefree} below.
Here $V_Q$ is the standard operator defined in Lemma~\ref{lem:uQ}.
\begin{theorem}\label{thm:UQVQ}
 Suppose that $\ell\geq 5$  and $Q\geq 5$ are primes with $Q\neq \ell$.
\begin{enumerate}
\item  There are no congruences of the form \eqref{eq:Qellcong} with 
$\pfrac{24\beta_{\ell, Q}-1}Q=0$.

\item  Fix $\delta\in \{0, -1\}$ and  $\ep\in \{\pm 1\}$.  If  there is a congruence of the form \eqref{eq:Qellcong}
with 
\begin{gather*}
 \pmfrac{1-24\beta_{\ell, Q}}\ell=\delta\ \ \text{and}\ \ \pmfrac{24\beta_{\ell, Q}-1}Q=\ep
\tx{,}
\end{gather*}
then
\begin{gather}\label{eq:uqvq}
f_{\ell, \delta}\big| U_Q\equiv -\ep\pmfrac{-12}Q Q^{-1} \,f_{\ell, \delta}\big|V_Q\pmod\ell.
\end{gather}

 \end{enumerate}
\end{theorem}

\begin{remark}

In terms of the partition function, \eqref{eq:uqvq} can be expressed in the form
\begin{gather*}
\sum_{\pfrac{-n}\ell=\delta} p\pmfrac{Qn+1}{24}q^\frac n{24}\equiv 
-\ep\pmfrac{-12}Q Q^{-1}\sum_{\pfrac{-n}\ell=\delta} p\pmfrac{n+1}{24}q^\frac {Qn}{24}
\pmod\ell.
\end{gather*}
\end{remark}

\begin{remark}
We can also deduce in Theorem~\ref{thm:UQVQ} that  $f_{\ell, \delta}$ is an eigenform$\pmod\ell$ 
of the Hecke operator  $T_{Q^2}$, with eigenvalue
\begin{gather*}
-\ep\pmfrac{12}Q\(Q^{-1}+Q^{-2}\)\pmod\ell.
\end{gather*}
However, this statement is much weaker than \eqref{eq:uqvq}.
\end{remark}

 Our last theorem gives 
 necessary and sufficient conditions for the existence 
of congruences~\eqref{eq:quadcong}.
The statement involves the twist $f_{\ell, \delta}\otimes \chi_Q$, which is defined in \eqref{eq:quadtwist}.

\Needspace*{3\baselineskip}
\begin{theorem}
\label{thm:UQQVQQ}
Suppose that $\ell\geq 5$ and $Q\geq 5$ are primes with $Q\neq\ell$.  
\begin{enumerate}
\item If\/ $Q^2\mid (24\beta_{\ell, Q}-1)$ then the only congruences of the form \eqref{eq:quadcong} arise from the three Ramanujan congruences.
\item If\/ $(Q, 24\beta_{\ell, Q}-1)=1$ then we have \eqref{eq:quadcong} if and only if~\eqref{eq:Qellcong}.
\item Fix $\delta\in \{0, -1\}$.
If\/ $Q\mid\mid (24\beta_{\ell, Q}-1)$ then we have a congruence  \eqref{eq:quadcong} 
with $\pfrac{1-24\beta_{\ell, Q}} \ell=\delta$
if and only if both of the following are true:
\begin{gather}\label{eq:Qiffone}
f_{\ell, \delta}\big|U_{Q^2}
\equiv \pmfrac{-12}QQ^{-1} 
f_{\ell, \delta}\otimes\chi_Q
+Q^{-2}f_{\ell, \delta}\big|V_{Q^2}
\pmod\ell
\end{gather}
and
\begin{gather}\label{eq:Qifftwo}
f_{\ell, \delta}\big|U_Q\equiv f_{\ell, \delta}\big|U_{Q^2}V_Q\pmod\ell.
\end{gather}
\end{enumerate}
\end{theorem}

Using Theorem~\ref{thm:UQQVQQ}, we are able to prove that there are no congruences 
\eqref{eq:quadcong} for small values of $\ell$ and $Q$.  In particular, we have the following
\begin{corollary}\label{cor:noquadcong}
There are no congruences \eqref{eq:quadcong} for $17\leq \ell<1,000$ and $5\leq Q<10,000$.
\end{corollary}
\begin{remark}
We verified that there are no congruences~\eqref{eq:quadcong} for~$\ell = 13$ and~$5\leq Q<10,000$ if~$\pfrac{1-24\beta_{\ell, Q}}\ell = \delta = 0$. The case~$\delta = -1$ evades our obstructions in analogy to the situation in Corollary~\ref{cor:nocase2}.
\end{remark}

We give a broad sketch of the arguments which we use to prove these results. 
The arguments are  technical in many places,  and they rely heavily on the theory
of modular forms  for powers of the eta-multiplier which is described in the next section.
We remark that many of our results will extend to a wider class of weakly holomorphic modular forms;
here we have focused on the prototypical case of the partition function
since  the technical difficulties which arise are already formidable.

In  Section~\ref{sec:squarefree} we prove a  generalization of Theorem~\ref{thm:UQVQ}
with squarefree modulus $Q$.  An important result of Radu (Theorem~\ref{thm:raduclasses} below)
shows that congruences \eqref{eq:Qellcong} are stable on square-classes of the parameter $1-24\beta_{\ell, Q}$.  This allows us to relate the existence of a single congruence
\eqref{eq:Qellcong} to properties    of the modular forms $f_{\ell, \delta}$.    If there is  a congruence \eqref{eq:Qellcong}, then for each $\beta$ along a square-class, we use another result of Radu to compute
the expansion of the weakly holomorphic modular form
$
g_\beta= q^\frac{24\beta-1}{24\ell Q}\sum p (\ell Q n + \beta) q^n
$
at a particular cusp of~$X(\ell Q)$.   Theorem~\ref{thm:UQVQ} follows from applying  the $q$-expansion principle as described in the next section and assembling the contributions from each $\beta$.

 Section~\ref{sec:proofmain}, which contains the proof of  Theorem~\ref{thm:main},
is the heart of the paper.
Let $S$ denote the set of primes for which there is a congruence.  
Using Theorem~\ref{thm:UQVQ} we identify two possible conditions which may hold at each member of a set of  auxiliary primes.
If the first condition holds for infinitely many  primes, we are able to  conclude that $S$ has density zero.
If the second holds for all but finitely many primes then a delicate argument involving Galois representations and 
the arithmetic large sieve shows that \eqref{eq:al0} holds.

In Section~\ref{sec:square}  we prove Theorem~\ref{thm:UQQVQQ} with methods which  are similar to those in Section~\ref{sec:squarefree}.
Here we are able to give necessary and sufficient conditions for the existence of a congruence.
Finally, Section~\ref{sec:comp} describes the computations which lead to Corollary~\ref{cor:nocase2}, Theorem~\ref{thm:nosmallcong}, and Corollary~\ref{cor:noquadcong}.

\section*{Acknowledgments}
We thank Nickolas Andersen, Alexander Dunn, Kevin Ford,  Olav Richter and Will Sawin for their helpful comments.


\section{Background}\label{sec:background}
The Dedekind eta function and the  theta function are defined by
\begin{gather*}
  \eta(\tau):=q^\frac1{24}\prod_{n=1}^\infty(1-q^n)
\quad\tx{and}\quad
  \theta(\tau) := \sum_{n=-\infty}^\infty q^{n^2}
\tx{,}
\end{gather*}
where we use the notation 
\begin{gather*}
q:=e(\tau)=e^{2\pi i\tau},\  \ \ \tau \in \mathbb{H}.
\end{gather*}
The  eta function is a modular form of weight $\frac{1}{2}$ on $\SL_2(\Z)$; in particular, there is a multiplier $\nu_\eta$ with 
\begin{gather}\label{eq:etamult}
\eta(\gamma\tau)=\nu_\eta(\gamma)(c\tau+d)^\frac12\,\eta(\tau), \qquad \gamma=\pMatrix abcd\in \SL_2(\Z)\tx{.}
\end{gather}
Here and throughout, we choose the principal branch of the square root. For $c>0$ we have the explicit formula \cite[\S 4.1]{knopp}:
\begin{gather} \label{eq:nueta}
  \nu_\eta(\gamma) = 
  \begin{dcases}
    \( \mfrac dc \) \, e\(\mfrac 1{24} \big( (a+d)c-bd(c^2-1)-3c \big) \)\tx{,}
  & \text{ if $c$ is odd}, \\
    \(\mfrac cd \) \, e\(\mfrac 1{24} \big((a+d)c-bd(c^2-1)+3d-3-3cd\big)\)\tx{,}
  & \text{ if $c$ is even.}
  \end{dcases}
\end{gather}
The multiplier for the theta function is given by
\begin{gather*}
  \nu_{\theta}\left( \begin{mmatrix} a & b \\ c & d \end{mmatrix} \right)
:=
  (c\tau + d)^{-\frac12} \, \frac{\theta(\gamma \tau)}{ \theta(\tau)}
=
  \pmfrac{c}{d} \ep_d^{-1},
\end{gather*}
where
\begin{gather*}
\ep_d = \begin{cases} 
1\tx{,} & d\equiv 1 \pmod{4} \tx{,}\\
i\tx{,} & d \equiv 3 \pmod{4} \tx{.}
\end{cases}
\end{gather*}
For odd values of $d$, $d_1$, $d_2$ we have the useful formulas 
\begin{alignat}{2}
\label{eq:oneminusd-epd1d2}
 e\pmfrac{1-d}8&= \pmfrac2d\ep_d
\quad\tx{and}\quad
  \ep_{d_1d_2} &= \ep_{d_1}\ep_{d_2}(-1)^{\frac{d_1-1}2\frac{d_2-1}2}
\tx{.}
\end{alignat}
Define the Gauss sum for odd $d$ by 
\begin{gather*}
G(a,d) := \sum_{n\spmod d }  \left( \frac{n}{d} \right) e\pmfrac{an}{d}\tx{.}
\end{gather*}
Write $G(d)=G(1, d)$ for simplicity and recall that $G(a, d)=\pfrac a d G(d)$ if $(a, d)=1$.
For $d$ odd and squarefree we have the evaluation
\begin{gather}\label{eq:gausssum}
G(d)=\ep_d\sqrt d.
\end{gather}

Following Shimura \cite{shimura} let  $G$ be the group of pairs $[\alpha, \varphi(\tau)]$ where $\alpha=\pmatrix abcd\in \GL_2^+(\R)$ and $\varphi$ is  a holomorphic function on $\H$ with $\varphi(\tau)^2=t(\det\alpha)^{-\frac12}(c\tau+d)$, where $|t|=1$. The group operation is given by 
$$
[\alpha, \varphi(\tau)] \cdot [\beta, \rho(\tau)] = [\alpha \beta, \varphi(\beta \tau) \rho (\tau) ].
$$
For $k\in \frac12\Z$, $G$ acts on holomorphic functions $f$ on $\H$ by 
\begin{gather*}
\(f\sk[\alpha, \varphi(\tau)]\)(\tau):=\varphi(\tau)^{-2k}f(\alpha\tau).
\end{gather*}
If $\gamma=\pmatrix abcd\in \SL_2(\Z)$ we define $\gamma^*:=[\gamma, (c\tau+d)^{\frac12}]\in G$.

Throughout the paper, $\ell\geq 5$ will denote a fixed prime number.
Given  $k\in \frac12\Z$,  a positive integer $N$, and a multiplier system $\nu$ on $\Gamma_0(N)$, we denote by $M_k\(N, \nu \)$, $S_k\(N,  \nu \)$, and $M^!_k\(N,  \nu \)$ the spaces of modular forms, cusp forms, and weakly holomorphic modular forms of weight $k$ and multiplier $\nu$ on $\Gamma_0(N)$ 
whose Fourier coefficients are  algebraic numbers which are integral at all primes above $\ell$.
Forms in these spaces satisfy the transformation law
\begin{gather*}
f\sk \gamma^*=\nu (\gamma) f\ \ \ \text{for}\ \  \  \gamma=\pmatrix abcd\in \Gamma_0(N)
\end{gather*}
as well as the appropriate conditions at the cusps of $\Gamma_0(N)$ (weakly holomorphic forms are allowed poles at the cusps).
We assume  familiarity with the  situation when   $\nu=\chi \nu_\theta^r$, where $r\in \Z$ and $\chi$ is  a Dirichlet character.

We  will be mostly concerned with the spaces   $M_k\(N,  \chi \nu_\eta^r\)$ where $(r, 24)=~1$ and $\chi$ is  a Dirichlet character modulo~$N$,
and we  summarize some of their important properties.
 If $f\in M_k\(N, \chi \nu_\eta^r\)$,
then $\eta^{- r}f \in M^!_{k-\frac r2}(N,\chi)$.
It follows that $f$ has a Fourier expansion of the form
\begin{gather}
\label{eq:fourier-expansion}
  f
=
  \sum_{n\equiv  r\spmod{24}}
  a(n)q^\frac n{24}
\tx{.}
\end{gather}
We also see that 
\begin{gather}\label{eq:krcond}
  M_k\(N, \chi \nu_\eta^r\)=\{0\}
  \quad\text{unless}\quad
  2k-r \equiv 1-\chi(-1) \pmod{4}
\tx{.}
\end{gather}
In  particular, the assumption that $(r,24) = 1$ implies that  $k\not\in\Z$.

We will frequently make use of the $U$ and $V$ operators, whose properties are summarized in the next lemma.
\begin{lemma}\label{lem:uQ}
Suppose that $(r, 24)=1$, that $f \in M_k\(N,\chi\nu_\eta^r\)$ has Fourier expansion~\eqref{eq:fourier-expansion},  and that $m$ is a positive integer. 
Define 
\begin{gather*}
  f\big|U_m:=\sum a(mn)q^\frac n{24}
\quad\tx{and}\quad
  f\big|V_m:=\sum a(n)q^\frac {mn}{24}
  \tx{.}
\end{gather*}
Then 
\begin{gather*}
  f \big|U_m=m^{\frac k2-1} \sum_{v\spmod m} f\sk\left[\pmatrix 1 {24 v}0m, m^\frac14\right]
\quad\tx{and}\quad
  f \big|V_m=m^{-\frac{k}{2}} f\sk\left[\pmatrix m001, m^{-\frac14}\right]
\tx{.}
\end{gather*}
 \end{lemma}
\begin{proof}
  Letting $\zeta_m$ denote a primitive $m$th root of unity, the sum over $v\pmod m$ becomes
\begin{gather*}
  m^{-1} \sum_n a(n)q^\frac{n}{24 m}\sum_{v\spmod m} \zeta_m^{v n}
=
  \sum_n a(mn)q^\frac n{24}.
\qedhere
\end{gather*}
\end{proof}
From a computation involving \eqref{eq:nueta} and \eqref{eq:oneminusd-epd1d2} it follows that for $(r, 24)=1$ we have 
\begin{gather}\label{eq:etamultv24}
  f \in M_k\(N, \chi\nu_\eta^r\)
\implies
  f\big| V_{24} \in M_k\(576N,  \chi\ptfrac{12}\bullet \nu_{\theta}^r \)
\tx{.}
\end{gather}
  If~$M$ is an odd positive squarefree integer,  denote by $\chi_M = \big(\frac{\bullet}{M}\big)$ the quadratic character of modulus~$M$.
Using Lemma~\ref{lem:uQ} and \eqref{eq:nueta} it can be checked 
(via  a somewhat tedious calculation which relies on \eqref{eq:oneminusd-epd1d2})
that if $(r, 24)=1$ and  $Q\geq 5$ is prime, then 
\begin{alignat}{2}
\label{eq:uqaction}
  U_Q &&\,:\, M_k\(N, \chi \nu_\eta^r\) &\longrightarrow M_k\(N \mfrac Q{(N, Q)},\, \chi\chi_Q\nu_\eta^{Qr}\)
\tx{,}
\\
\label{eq:vqaction}
  V_Q &&\,:\, M_k\(N, \chi\nu_\eta^r\) &\longrightarrow M_k\(NQ, \,\chi\chi_Q\nu_\eta^{Qr}\)
\tx{.}
\end{alignat}

 If $Q\geq 5$ is prime
 and $(r, 24)=1$,  we define the twist of~$f \in M_k(N,\chi \nu_\eta^r)$ with Fourier expansion~\eqref{eq:fourier-expansion} by
\begin{gather}\label{eq:quadtwist}
  f\otimes \chi_Q:=\sum \chi_Q (n) a(n)q^\frac n{24}
\tx{.}
\end{gather}
(this normalization, which disregards the denominator of the exponents,  is chosen for ease of notation in the proofs).
We have
\begin{gather*}
  f\otimes \chi_Q=\mfrac1{G(Q)}\sum_{v\spmod Q} \chi_Q (v) f\sk \left[\pmatrix1{\frac{24 v}Q}01, 1\right]
\tx{,}
\end{gather*}
from which a computation  as above gives
\begin{gather*}
  f\otimes\chi_Q \in M_k\(NQ^2, \chi\nu_\eta^r\)
\tx{.}
\end{gather*}

For each prime $Q\geq 5$ we have the Hecke operator
\begin{gather*}
  T_{Q^2} \,:\, S_k\(1, \nu_\eta^r\)\rightarrow S_k\(1, \nu_\eta^r\)
\tx{.}
\end{gather*}
If $f \in S_k\(1, \nu_\eta^r\)$  with $(r, 24)=1$ has Fourier expansion~\eqref{eq:fourier-expansion} then we have (see for example \cite[Proposition~11]{Yang_shimura})
\begin{multline}\label{eq:heckedef}
  f\big|T_{Q^2}
=
  \sum \(
  a(Q^2n) + Q^{k-\frac32}\pmfrac{-1}{Q}^{k-\frac12}\pmfrac{12n}{Q} a(n) + Q^{2k-2}a\pmfrac{n}{Q^2}
  \)
  q^\frac n{24}\\
  =f\big|U_{Q^2}+Q^{k-\frac32}\pmfrac{-1}{Q}^{k-\frac12}\pmfrac{12}{Q}f\otimes\chi_Q+Q^{2k-2} f\big|V_{Q^2}
\tx{.}
\end{multline}

For each squarefree $t$ with $(t, 6)=1$  there is a Shimura lift $\Sh_t$ on $S_k\(1, \nu_\eta^r\)$, defined via the relationship~\eqref{eq:etamultv24} and the usual Shimura lift~\cite{shimura} on~$S_k\(576, \pfrac{12}\bullet \nu_\theta^r\)$.
The lift  $\Sh_t$ can be described by  its action on Fourier expansions:
\begin{gather}\label{eq:shimlift}
\Sh_t\(\sum a(n)q^\frac n{24}\)=\sum A_t(n)q^n,
\end{gather}
where the coefficients $A_t(n)$ are given by 
\begin{gather}\label{eq:shimliftcoeff}
  A_t(n)
=
  \sum_{d\mid n}
  \pmfrac{-1}d^{k-\frac12}\pmfrac{12t}d d^{k-\frac32}\,
  a\pmfrac{tn^2}{d^2}
\tx{.}
\end{gather}
It follows that 
\begin{gather}\label{eq:shimcong}
  f \equiv 0 \pmod\ell
\iff
  \Sh_t(f) \equiv 0 \pmod\ell
  \qquad \text{for all squarefree $t$.}
\end{gather}
For the non-obvious direction of this equivalence, we argue as follows: if $f \not\equiv 0 \pmod\ell$ then for some squarefree~$t$ there is an  index~$n$ with $a(tn^2) \not\equiv 0\pmod\ell$. Letting $n_0$ denote the minimal such~$n$, we see from~\eqref{eq:shimliftcoeff} that $A_t(n_0) \not\equiv 0 \pmod\ell$.

The work of Shimura and Niwa \cite{Niwa} shows that if $f\in S_k\(1, \nu_\eta^r\)$, then $\Sh_t (f)\in S_{2k-1}(288)$. 
Moreover, for all primes $Q\geq 5$ we have 
\begin{gather}\label{eq:shimliftcommute}
  \Sh_t\(f\big|T_{Q^2}\)=\(\Sh_t f\)\big|T_Q
\tx{,}
\end{gather}
where $T_Q$ is the Hecke operator of index $Q$ on the integral weight space.

From recent  work of Yang~\cite{Yang_shimura}
it follows  that
\begin{gather}\label{eq:shimliftimage}
  \Sh_t \,:\,
  S_k\(1, \nu_\eta^r\)
\longrightarrow
  S_{2k-1}^{\operatorname{new}}(6)\otimes \pmfrac{12}\bullet
:=
  \left\{
  f \otimes \pmfrac{12}\bullet \,:\, f \in S_{2k-1}^{\operatorname{new}}(6)
  \right\}
\tx{.}
\end{gather}
To establish this, it suffices to prove that if $f\in S_k\(1, \nu_\eta^r\)$ is an eigenform of  $T_{Q^2}$ for primes $Q\geq 5$, then $\Sh_t(f)$ is in the space described in \eqref{eq:shimliftimage}.
Suppose that $f$ is such an eigenform and that $t$ is a squarefree positive integer with $(t, 6)=1$ and $\Sh_t(f)\neq 0$.
By \cite[Thm~1]{Yang_shimura}, there is a  newform $F\in S_{2k-1}^{\operatorname{new}}(6)$ with the same Hecke eigenvalues as   $\Sh_t(f)\otimes \pfrac{12}\bullet$ at all primes $Q\geq 5$.
By strong multiplicity one~\cite{jacquet-shalika-1981a,jacquet-shalika-1981b}, it follows that $\Sh_t(f)\otimes \pfrac{12}\bullet$ is a constant multiple of $F$.  The claim follows since
 $\Sh_t(f)$ is supported on exponents coprime to $6$.


If $M$ is a positive integer, we define
\begin{gather*}
W_M:=\left[\pmatrix 0{-1}M0, M^\frac14\tau^\frac12\right]
\tx{.}
\end{gather*}
We require the following  version of the $q$-expansion principle (see \cite[VII, Corollary~3.12]{deligne-rapoport} or \cite[Theorem~4.8]{Radu_aoconj}).

\begin{proposition}\label{prop:delignerap}
Suppose that $k$ and $N$ are positive integers, that $\ell$ is prime, and that $\pi$ is prime ideal above~$\ell$ in a number field~$\mathbb{F}$ which  contains all $N$th roots of unity. Write~$\mathcal{O}_\pi \subseteq~\mathbb{F}$ for the ring of elements which are integral at $\pi$. Suppose that $f\in M_k(\Gamma(N))\cap \mathcal{O}_\pi[\![q^\frac1N]\!]$ and that
 $\gamma \in \Gamma_0(\ell^m)$,
where $\ell^m$ is  the highest power of $\ell$ dividing $N$.
Then~$f\sk\gamma^* \in \mathcal{O}_\pi[\![q^\frac1N]\!]$, and for $n\geq 0$ we have 
\begin{gather*}
f \equiv 0 \pmod{\pi^n} \iff f\sk\gamma^* \equiv 0\pmod{\pi^n} \tx{.}
\end{gather*}
\end{proposition}

For convenience, we record a  lemma which allows us to apply Proposition~\ref{prop:delignerap} in a   straightforward way to weakly holomorphic modular forms of half-integral weight.
\begin{lemma}
\label{la:delignerap}
Suppose that $k \in \frac{1}{2} \Z$, that $r \in \Z$, and that $\ell$ is a prime number. Write~$\mathcal{O}_\ell \subseteq \overline{\mathbb{Q}}$ for the ring of algebraic numbers which are integral at all primes dividing~$\ell$. Let~$h$ be a positive integer. Suppose that~$f \in \overline{\mathbb{Q}}[\![q^\frac1{hN}]\!][q^{-1}]$ converges absolutely and locally uniformly if $0 < |q| < 1$, that $f^h \in M^!_{hk}(\Gamma(N),\nu_\eta^{r})\cap \mathcal{O}_\ell[\![q^\frac1N]\!][q^{-1}]$ or $f^h \in M^!_{hk}(\Gamma(N), \nu_\theta^{r})\cap \mathcal{O}_\ell[\![q^\frac1N]\!][q^{-1}]$, and that~$\gamma \in \Gamma_0(\ell^m)$, where $\ell^m$ is the highest power of~$\ell$ dividing~$N$. Then $f \sk\gamma^* \in \mathcal{O}_\ell[\![q^\frac1N]\!][q^{-1}]$, and for~$n \ge 0$ we have
\begin{gather}
\label{eq:la:delignerap}
  f \equiv 0
  \pmod{\ell^n}
\iff
  f \sk\gamma^* \equiv 0
  \pmod{\ell^n}
\tx{.}
\end{gather}
If~$\ell$ and~$N$ are coprime, then for any positive integer~$M$ with~$\ell \nmid M$ we have
\begin{gather}
\label{eq:qexp}
  f \equiv 0
  \pmod{\ell^n}
\iff
  f \sk W_M \equiv 0
  \pmod{\ell^n}
\tx{.}
\end{gather}
\end{lemma}
\begin{proof}
Since $W_M=\pmatrix0{-1}10^*\left[\pmatrix M001, M^{-\frac14}\right]$, the
 second assertion follows from  \eqref{eq:la:delignerap}.

Assume that the lemma holds for~$h=1$. Given~$f$ and an arbitrary~$h$ as in the statement, apply the lemma with~$n$ replaced by~$h n$. This yields~$f^h \big|_{hk} \gamma^* \in \mathcal{O}_\ell[\![ q^{\frac{1}{N}} ]\!][q^{-1}]$ and the equivalence
\begin{gather*}
  f^h \equiv 0
  \pmod{\ell^{hn}}
\iff
  f^h \big|_{hk}\gamma^* \equiv 0
  \pmod{\ell^{hn}}
\tx{.}
\end{gather*}
Observe that~$f^h \big|_{hk}\gamma^* = (f \sk \gamma^*)^h$, and hence $f \sk\gamma^* \in \mathcal{O}_\ell[\![ q^{\frac{1}{hN}} ]\!][q^{-1}]$. Since we further have the equivalences
\begin{alignat*}{3}
  &f^h \equiv 0
  &\pmod{\ell^{hn}}
&\iff
  f \equiv 0
  &&\pmod{\ell^n},
\\
  &f^h \big|_{hk}\gamma^* \equiv 0
  &\pmod{\ell^{hn}}
&\iff
  f \sk\gamma^* \equiv 0
  &&\pmod{\ell^n}
\tx{,}
\end{alignat*}
we can and will assume that~$h = 1$ in the remainder of the proof.

To prove~\eqref{eq:la:delignerap} in the case of~$h = 1$, we employ an argument of Jochnowitz~\cite{jochnowitz-2004-preprint}; a slightly different argument was given by Radu~\cite{Radu_aoconj}. We restrict to the case of the eta multiplier, since the theta multiplier can be handled in an analogous way. Further, we can and will replace~$\mathcal{O}_\ell$ by its intersection with a suitable number field that contains the Fourier coefficients of~$f$ and all~$N$th roots of unity. Let $\ell^n = \pi_1^{n_1} \cdots \pi_m^{n_m}$ be the prime ideal factorization of~$\ell^n$ in~$\mathcal{O}_\ell$. It suffices to show that
\begin{gather*}
  f \equiv 0
  \pmod{\pi_i^{ n_i}}
\iff
  f \sk\gamma^* \equiv 0
  \pmod{\pi_i^{ n_i}}\ \ \ \text{for all $i$}
  \tx{.}
\end{gather*}
Observe that~$f \equiv 0 \pmod{\pi_i^{ n_i}}$ is equivalent to~$\eta^j f \equiv 0 \pmod{\pi_i^{n_i}}$ for any~$j \in \Z$. 
 Choosing a sufficiently large $j$ with   $j\equiv -r\pmod{24}$, we have 
 $\eta^j f \in M_{k + \frac j 2}(\Gamma(N))$.  The lemma follows by applying  Proposition~\ref{prop:delignerap} to $\eta^j f$.
\end{proof}

Finally, we justify the statements~\eqref{eq:fldelprop} and~\eqref{eq:fldeldef} from the introduction. Let~$\Delta = \eta^{24}$ be the unique normalized cusp form of weight $12$ on 
$\SL_2(\Z)$, and   let $U_\ell$ and $V_\ell$ denote the usual operators on spaces of integral weight modular forms.
By~\cite[(3.2)]{Ahlgren-Boylan} we have 
\begin{gather*}
  \sum_{\pfrac{-n}\ell=0}
  p\pmfrac{n+1}{24} q^{\frac{n}{24}}
\equiv
   \pmfrac{\Delta^{\frac{\ell^2-1}{24}}\big| U_\ell}{\eta^\ell}  \Big|V_\ell
  \pmod\ell
\tx{.}
\end{gather*}
By \cite[Lemma 2]{serre_formes}, $\Delta^\frac{\ell^2-1}{24}\big| U_\ell$ is congruent modulo~$\ell$ to a modular form $G_\ell$ of weight  $\ell-1$ on $\SL_2(\Z)$. 
The form $G_\ell$ vanishes modulo $\ell$ to order  $>\frac\ell{24}$.  Note that  $M_{\ell-1}(1)\cap\Z[\![q]\!]$ has a basis $\{h_1, \dots, h_d\}$ with $h_j=q^j+\dots$.  After subtracting a suitable integral linear combination of 
these basis elements, we may assume that $G_\ell$ vanishes to order $>\frac\ell{24}$.
Therefore we can take 
\begin{gather*}
F_\ell= \mfrac{G_\ell}{\eta^\ell} \in S_{\frac{\ell - 2}{2}} \( 1, \nu_\eta^{-\ell}\)
\end{gather*}
and 
$f_{\ell,0}=F_\ell^\ell \in S_{\frac{\ell^2 - 2\ell}{2}} \( 1, \nu_\eta^{-1}\)$ as claimed. After dividing the modular form which is described in  \cite[(3.4)]{ahlgren-ono_conj} by $\eta^{\ell^2}$, we deduce that there is a modular form $g\in S_{\frac{\ell^2 - 2}{2}} \( 1, \nu_\eta^{-1}\)$ such that 
\begin{gather*}
  g
\equiv
  \sum_{\pfrac{-n}\ell=0}
  p\pmfrac{n+1}{24} q^{\frac{n}{24}}
  +
  2  
  \sum_{\pfrac{-n}{\ell}=-1}
  p\pmfrac{n+1}{24} q^{\frac{n}{24}}
  \pmod\ell
\tx{.}
\end{gather*}
Subtracting $f_{\ell, 0} E_{\ell-1}$, where $E_{\ell-1}\equiv 1\pmod \ell$ is the Eisenstein series of weight $\ell-1$, it follows that 
 we can also take~$f_{\ell,-1}$ as stated.

\section{Proof of Theorem~\ref{thm:UQVQ}}\label{sec:squarefree}
In this section we address the issue of congruences $p(\ell Q n+\beta)\equiv 0\pmod\ell$ for squarefree integers $Q$.
Theorem~\ref{thm:UQVQ} will follow from Proposition~\ref{prop:killsillyQ} and Theorem~\ref{thm:theorem1full} 
below.

Given a positive integer $m$ and an integer $\beta$, define the set 
\begin{gather}\label{eq:pmtdef}
S_{m, \beta}:=\{\beta'\!\!\pmod m\ : \ 24\beta'-1\equiv a^2(24\beta-1)\!\!\pmod m\ \ \text{for some $a$ with $(a, 6m)=1$}\}.
\end{gather}
Crucial to our arguments is  the following result of Radu.
\begin{theorem}{\cite[Theorem~5.4]{Radu_subbarao}}\label{thm:raduclasses}
Suppose that $m$ and $\ell$ are positive integers and that for some integer $\beta$  we have a congruence $p(mn+\beta)\equiv 0\pmod\ell$.
Then for all $\beta'\in S_{m, \beta}$ we have a congruence $p(mn+\beta')\equiv 0\pmod\ell$.
\end{theorem}

We begin by making a  reduction.
From \cite{Ahlgren-Boylan} we know that for $\ell\geq 13$  there are no congruences of the form
\begin{gather*}
p(\ell n+\beta)\equiv 0\pmod\ell.
\end{gather*}
It follows from the next result that for such $\ell$  there are no congruences of the form
\begin{gather*}
p(\ell Qn+\beta)\equiv 0\pmod\ell
\tx{,}
\end{gather*}
where $Q$ is a squarefree positive integer with $Q\mid  24\beta-1$.
In particular,  this result implies  the first assertion in Theorem~\ref{thm:UQVQ}.

\begin{proposition}\label{prop:killsillyQ}
Suppose that $\ell\geq 5$ is prime,   that  $Q$ is a squarefree positive integer with  $(Q, 6\ell)=1$, and that $\beta\in \Z$.  
Write 
\begin{gather}\label{eq:Qprimedef}
Q = Q' Q'', \ \ \text{where}\ \  (Q',  24\beta-1) = 1 \ \ \text{and}\ \  Q'' \mid  24 \beta-1.
\end{gather}
Then there is a congruence 
\begin{gather}\label{eq:cong1}
p(\ell Qn+\beta)\equiv 0\pmod\ell
\end{gather} 
if and only if there is a congruence 
\begin{gather}\label{eq:cong2}
p(\ell Q'n+\beta)\equiv 0\pmod\ell.
\end{gather} 
\end{proposition}

\begin{proof}  
For the non-obvious direction, suppose that there is a congruence \eqref{eq:cong1}.
Set 
\begin{gather*}
\delta:=\pmfrac{1-24\beta}\ell\in \{0, -1\}
\end{gather*}
 and let  $Q'=Q_1\cdots Q_t$ be the prime factorization of $Q'$.
For $i=1,\dots, t$ set
\begin{gather*}
\delta_i:=\pmfrac{24\beta-1}{Q_i}\in \{\pm1\}.
\end{gather*}

Suppose that  $f=\sum a(n)q^\frac n{24}\in S_k\(N, \nu_\eta^{-1}\)$.
For each $i$, define the operator 
\begin{gather*}
f\big|  B_i:=  \mfrac12\(f-f \big| U_{Q_i}V_{Q_i} + \delta_i \, f\otimes \chi_{Q_i}\).
\end{gather*}
 Then we have
\begin{gather*}
f\big|  B_i=\sum_{\pfrac n{Q_i}=\delta_i}a(n)q^\frac n{24}
\in S_k\(NQ_i^2, \nu_\eta^{-1}\).
\end{gather*}
Recall the definition \eqref{eq:fldeldef}, and define
\begin{gather*}
g:=f_{\ell, \delta}\big| B_{Q_1}\dots B_{Q_t}
\in S_k\((Q')^2, \nu_\eta^{-1}\)
\end{gather*}
(where the particular value of $k$ depends on $\delta$ and is unimportant in what follows).  We have
\begin{gather*}
g\equiv \sum_{\substack{\pfrac{-n}\ell=\delta\\ \pfrac{n}{Q_i}=\delta_i \,\forall i}} p\pmfrac{n+1}{24}q^\frac n{24}\pmod\ell.
\end{gather*}

After changing variables we find that 
\begin{gather*}
g\equiv \sum_{\beta'\in S_{\ell Q', \beta}}\sum_n p(\ell Q' n+\beta')q^{\ell Q' n+\beta'-\frac1{24}}\pmod\ell.
\end{gather*}
So by  Theorem~\ref{thm:raduclasses}, we have the congruence \eqref{eq:cong2} if and only if 
$g\equiv 0\pmod\ell$.

Let  $Q''=P_1\dots P_s$ be the prime factorization of $Q''$, and define 
\begin{gather*}
h:=g\big| U_{P_1}\dots U_{P_s}V_{P_1}\dots V_{P_s}\equiv \sum_{\substack{\pfrac{-n}\ell
= \delta\\ \pfrac{n}{Q_i}=\delta_i \,\forall i\\ P_j | n \,\forall j}} p\pmfrac{n+1}{24}q^\frac n{24}\pmod\ell .
\end{gather*}
Then 
\begin{gather*}
h\equiv \sum_{\beta'\in S_{\ell Q, \beta}}\sum_n p(\ell Q n+\beta')q^{\ell Q n+\beta'-\frac1{24}}\pmod\ell,
\end{gather*}
so we have the congruence \eqref{eq:cong1} if and only if $h\equiv 0\pmod\ell$.

Proposition~\ref{prop:killsillyQ}  therefore follows from the next lemma, applied successively with the primes $P_1, \dots, P_s$.
\end{proof}

\begin{lemma}\label{lem:uq}
Suppose that $\ell\geq 5$ is prime, that   $(r, 24)=1$, that  $(N, 6\ell)=1$,   that $Q$ is a prime with $(Q, 6N\ell)=1$, and that $\chi$ is a Dirichlet character modulo $N$.
  Suppose that  $f\in M_k\(N, \chi\nu_\eta^r\)$.
Then 
\begin{gather*}
f\big|U_Q\not\equiv 0\pmod\ell\iff f\not\equiv 0\pmod\ell.
\end{gather*}
\end{lemma}
\begin{proof}[Proof of Lemma~\ref{lem:uq}]
It suffices to prove the assertion for the modular form $f\big|V_{24}$.
Using~\eqref{eq:etamultv24} and setting $N_1=576N$,  we will  assume for simplicity in the proof that
\begin{gather*}
  f \in 
  M_k\(N_1, \ptfrac{12}{\bullet}\chi \nu_\theta^r\)
\tx{.}
\end{gather*}
Only one direction requires proof.
Since $f$ has a Fourier expansion in integral powers of $q$, we have 
\begin{gather}\label{eq:lem1}
f\big|U_Q\sk W_{QN_1}=Q^{\frac k2-1}\sum_{v\spmod Q}f\sk\left[\pmatrix{vQN_1}{-1}{Q^2N_1}0, Q^\frac12N_1^\frac14\tau^\frac12\right].
\end{gather}
Write 
\begin{gather*}
g:=f\sk W_{N_1}=\sum b(n)q^n
\end{gather*}
 for the image under the Fricke involution~$W_{N_1}$. 
The term with $v=0$ in \eqref{eq:lem1} gives 
\begin{gather}\label{eq:lem2}
Q^{\frac k2-1}f\sk W_{N_1}\left[\pmatrix {Q^2}001, Q^{-\frac12}\right] =Q^{\frac{3k}2-1} g\big|V_{Q^2}.
\end{gather}
For each  $v\not\equiv 0\pmod Q$ in \eqref{eq:lem1} we choose $\alpha$ with 
\begin{gather*}
 \alpha v N_1\equiv 1\pmod Q.
\end{gather*}
Then the terms in \eqref{eq:lem1} with $v\not\equiv 0\pmod Q$  contribute
\begin{gather*}
Q^{\frac k2-1}\sum_{v\spmod Q^*}f\sk W_{N_1} \pmatrix{Q}\alpha{-vN_1}{\frac{1-\alpha v N_1}Q}^*\left[ \pmatrix {Q}{-\alpha}0{Q}, 1\right].
\end{gather*}
To compute the automorphy factors here and below it is helpful to use the facts that
\begin{gather*}
\ptfrac zw^\frac12=\tfrac{z^\frac12}{w^\frac12}, \ \ \ \ \  (-z)^\frac12=-iz^\frac 12\ \ \ \tx{for $z$, $w\in \H$}.
\end{gather*}
Since
\begin{gather*}
  W_{N_1}
  \left(\begin{smallmatrix}
  Q & \alpha \\
  -vN_1 & \frac{1 -  \alpha v N_1}{Q}
  \end{smallmatrix}\right)^*
=
  \left(\begin{smallmatrix}
  \frac{1 -  \alpha v N_1}{Q} & v \\
  - N_1 \alpha & Q
  \end{smallmatrix}\right)^*W_{N_1}
\end{gather*}
and $r$ is odd, this simplifies to
\begin{multline*}
  Q^{\frac{k}{2}-1}\chi(Q)\pmfrac{12N_1 }{Q}\ep_Q^{-r}
  \sum_{v \,(Q)^\ast}
  \pmfrac{- \alpha}{Q}\,
  g\big( \tau - \mfrac{\alpha}{Q} \big)
\\
=
  Q^{\frac{k}{2}-1}\chi(Q)
  \pmfrac{12N_1 }{Q}\ep_Q^{-r}\,
  \sum_n b(n) q^n
  \sum_{v \,(Q)^\ast}
  \pmfrac{-\alpha}{Q}
  \zeta_Q^{-n \alpha}
=
  Q^{\frac{k}{2}-1}\chi(Q)
  \pmfrac{12N_1}{Q} \ep_Q^{-r}
  G(Q) \,
  g \otimes \chi_Q
\tx{.}
\end{multline*}

Summarizing, we have
\begin{gather*}
  f \big| U_Q \sk W_{QN_1}
=
  Q^{\frac{k}{2}-1}\chi(Q)
  \pmfrac{12N_1}{Q} \ep_Q^{-r}
  G(Q) \,
  g \otimes \chi_Q
  +
  Q^{\frac{3k}{2}-1} g \big| V_{Q^2}
\tx{.}
\end{gather*}
 By~\eqref{eq:qexp}, we have $f|U_Q\equiv 0 \pmod\ell$ only if this expression is zero modulo $\ell$, which can happen only if  $g\equiv 0\pmod\ell$. By another application of \eqref{eq:qexp}, this can occur only if $f\equiv 0 \pmod\ell$.
\end{proof}

The second assertion of Theorem~\ref{thm:UQVQ}  
 follows from the next result.
\begin{theorem}\label{thm:theorem1full} 
Suppose that $\ell\geq 5$ is prime and that $Q$ is a squarefree positive integer with $(Q, 6\ell)=1$
and prime factorization $Q=Q_1\dots Q_t$.
Suppose that $\beta_0\in \Z$ has 
\begin{gather*}
(Q, 24\beta_0-1)=1
\end{gather*}
and that there is a congruence
\begin{gather*}
p(\ell Q n + \beta_0) \equiv 0 \pmod\ell.
\end{gather*}
Define 
\begin{gather*}
\delta:=\pmfrac{1-24 \beta_0}{\ell}\in\{0, -1\},
\end{gather*}
and for   $d \mid Q$ define
\begin{gather}\label{eq:lambdadef}
\lambda_d := d^{-1} \pmfrac{-12}{d}  \pmfrac{24\beta_0 - 1 }{d}.
\end{gather}
Then we have 
\begin{gather}\label{eq:felldelQt}
f_{\ell,\delta} \big| \left(U_{Q_1} + \lambda_{Q_1} V_{Q_1} \right) \cdots \left(   U_{Q_t} + \lambda_{Q_t}V_{Q_t} \right)  \equiv 0 \pmod\ell .
\end{gather}
\end{theorem}

To prove Theorem~\ref{thm:theorem1full}, we require another result of Radu \cite{Radu_aoconj}.
Suppose that $m$ is a positive integer with $(m,6)=1$, and that $\beta\in \Z$.
As in \cite[Lemma~4.11]{Radu_aoconj}, define 
\begin{gather}\label{eq:gmdef}
g(m,\beta,\tau) := q^\frac{24\beta-1}{24m}\sum p (mn + \beta) q^n.
\end{gather}
If $\ell\geq 5$ is prime and $Q$ is a positive integer with~$(Q, 6\ell)=1$, we  choose integers $X$ and $Y$ with 
\begin{gather}\label{eq:Xcong}
576 \ell^2 X + QY = 1
\end{gather}
and define 
\begin{gather}\label{eq:gammadef}
\gamma_{\ell, Q}: = \pmatrix{1}{-24^2\ell X}{  \ell}{ QY}\in \SL_2({\Z}).
\end{gather}

The next proposition follows from Lemma~5.1 of~\cite{Radu_aoconj}.  Note that we have corrected a typographical error in that lemma which arises 
from dropping the  term $\pfrac{24\ell}{Q/d}$ in equation~(44).  The correct version, which we quote below, can also be found in equation (12) of  the preprint version 
available at the author's homepage.  Note also that for the values of $Q$ and $d$ below we have $(-1)^{\frac{Q+1}2\frac{d-1}2}=(-1)^{\frac{Qd-1}2\frac{d-1}2}$.
\begin{proposition}\label{prop:RaduTransformation}
Suppose that  $\ell\geq 5$ is prime, that $Q$ is a positive integer with $(Q, 6\ell)=1$, and that $\beta\in \Z$.
Let $\gamma_{\ell, Q}$ be defined as in \eqref{eq:gammadef}.  Then

\begin{gather}\label{eq:Radutransformation}
\begin{aligned}
&
  Q
  e \left( -\mfrac{ \pi i Q} {12} \right) e \left( - \mfrac{48 \pi i X (24 \beta -1)}Q \right)
  (\ell \tau + QY)^{\frac{1}{2}}
  g(\ell Q, \beta, \gamma_{\ell, Q} \tau)
\\
={}&
  \sum_{d \mid Q}
  d^{-\frac{1}{2}}
  e\left( \mfrac{1-d}{8} \right)
  \pmfrac{24 \ell}{Q/d}
 (-1)^{\frac{Qd-1}{2} \frac{d-1}{2}}
  q^{\frac{(24 t_{d, \beta} -1) d^2}{24  \ell Q}}
  \sum_{n=0}^{\infty}
  q^{\frac{nd^2}{Q}}
  p(\ell n + t_{d, \beta})
  T(n,d)
\tx{,}
\end{aligned}
\end{gather} 
where
  $t_{d, \beta}$ is the integer satisfying  $0 \le t_{d, \beta} \le \ell -1$ and 
\begin{gather}\label{eq:tddef}
d^2(24 t_{d, \beta} -1) \equiv 24 \beta - 1 \pmod{\ell},
\end{gather}
$\overline{s}$ is any integer such that $s \overline{s} \equiv 1 \pmod{Q/d}$,
and
\begin{gather*}
T(n,d):= \sum_{ s \spmod{Q/d} }  \pmfrac{24 \ell s}{Q/d} e \left( \mfrac{-48 \pi i X}{Q/d} \big( \overline{s} (24 (\ell n +t_{d, \beta}) - 1) + s(24\beta-1)\big) \right).
\end{gather*}
\end{proposition}

Note that the definition of $T(n, d)$ depends implicitly on $Q$ and that we have $T(n, Q)=1$ for all $n$ (since $\pfrac01=1$).
Define the  Salie sum by
\begin{gather}\label{eq:salie}
S(a,b,c) := \sum_{n\spmod c } \left( \frac{n}{c} \right) e\pmfrac{an + b\overline{n}}{c}.
\end{gather}
Replacing~$- \overline{24} \overline{\ell}^2 s$ by~$s$ and using \eqref{eq:Xcong}, we see that 
\begin{gather}\label{eq:tndsalie}
\begin{aligned}
  T(n,d)
&{}=
  \pmfrac{- \ell}{Q/d}
  \sum_{s \spmod{Q/d} }
  \pmfrac{s}{Q/d} e\left( \mfrac{2 \pi i}{Q/d} \big( \overline{24}^2 \overline{\ell}^4 \overline{s} (24 (\ell n +t_{d, \beta}) - 1) + s (24\beta-1)\big) \right)
\\
&{}=
  \pmfrac{-\ell}{Q/d}
  S\(24\beta - 1, \overline{24}^2 \overline{\ell}^4 (24(\ell n + t_{d, \beta}) - 1), Q/d\)
\tx{.}
\end{aligned}
\end{gather}

\begin{proof}[Proof of Theorem~\ref{thm:theorem1full}]
Let $\ell$ and $Q=Q_1 \dots Q_t$ be  as in the statement of the theorem, and suppose that there is a congruence  
\begin{gather*}
p(\ell Qn + \beta_0) \equiv 0 \pmod\ell\ \ \ \text{with} \ \ \ \pmfrac{1-24 \beta_0}{\ell}=\delta\in \{0, -1\}.
\end{gather*}
 By Theorem~\ref{thm:raduclasses} we have congruences
 \begin{gather}\label{eq:allbeta}
 p(\ell Q n + \beta)\equiv 0 \pmod{\ell}\ \ \text{for all}\ \ \beta \in S_{\ell Q,\beta_0}.
 \end{gather}
Suppose that $\beta \in S_{\ell Q,\beta_0}$. 
By equations (17) and (18) of \cite{Radu_aoconj},  there exists a positive integer $h$ such that 
$g(\ell Q, \beta,  \tau)^h\in M^!_\frac h2 (\Gamma(\ell Q))$.
Applying Lemma~\ref{la:delignerap} to $g(\ell Q, \beta,  \tau) \equiv 0 \pmod{\ell}$, we conclude that
\begin{gather}\label{eq:giszero}
(\ell \tau + YQ)^{\frac{1}{2}}  g(\ell Q, \beta, \gamma_{\ell, Q} \tau) \equiv 0 \pmod{\ell}.
\end{gather}
Therefore, the expression on the right side of \eqref{eq:Radutransformation} is $0 \pmod{\ell}$.
We write this expression in the form 
\begin{gather*}
\sum_{r \in \frac{1}{24\ell}\mathbb{Z}} a(r) q^r + \sum_{r \in \Q \backslash \frac{1}{24\ell}\mathbb{Z}} a(r) q^r,
\end{gather*}
and note that each of these summands is $0 \pmod\ell$.
Define
\begin{gather*}
  F_{ \ell, Q, \beta} := \sum_{r \in \frac{1}{24\ell}\Z} a(r) q^{ r} \equiv 0 \pmod{\ell}
\tx{.}
\end{gather*}
  
The theorem will follow from computing each $F_{ \ell, Q, \beta}$ explicitly.
In the term arising from the divisor $d$ in~\eqref{eq:Radutransformation}, the exponents have the form
\begin{gather}
\label{eq:Fbetaexponents}
r=\mfrac{nd^2}{Q} + \mfrac{(24 t_{d, \beta} - 1)d^2}{24 \ell Q} = \mfrac{d(24 (\ell n + t_{d, \beta}) - 1)}{24 \ell Q/d}.
\end{gather}
Since $(d,Q/d) = 1$, it follows that 
 $F_{ \ell, Q, \beta}$ is the sum of those terms in \eqref{eq:Radutransformation} with
 \begin{gather}\label{eq:fbetaterms}
 24 (\ell n + t_{d, \beta}) - 1 \equiv 0 \pmod{Q/d}. 
\end{gather}
To compute $F_{ \ell, Q, \beta}$, we may  therefore assume that  \eqref{eq:fbetaterms} holds.
Since $(24\beta-1, Q)=1$,  \eqref{eq:tndsalie} gives 
\begin{gather*}
T(n,d)= \pmfrac{ -\ell }{Q/d}   \pmfrac{24\beta - 1}{Q/d}  G\(Q/d\).
\end{gather*}
We  write
\begin{gather}\label{eq:fbetaddef}
F_{ \ell, Q, \beta} = \sum_{d \mid Q} F_{\ell,Q,\beta,d}
\tx{,}
\end{gather}
where for each divisor $d$, $F_{\ell,Q,\beta,d}$ is the contribution   from those terms in~\eqref{eq:Radutransformation}
satisfying~\eqref{eq:fbetaterms}.
In other words, 
\begin{multline*}
 F_{\ell,Q,\beta,d}=
  d^{-\frac{1}{2}} e\pmfrac{1-d}8
 (-1)^{\frac{Qd-1}{2} \frac{d-1}{2}} G(Q/d)
   \pmfrac{-24(24\beta-1)}{Q/d}\\
 \cdot\, \sum_{24(\ell n+t_{d, \beta}) - 1 \equiv 0 \,(Q/d)}
  p(\ell n + t_{d, \beta})  q^{\frac{nd^2}{Q}+\frac{(24 t_{d, \beta} -1) d^2}{24 Q \ell}}.
\end{multline*}
From~\eqref{eq:oneminusd-epd1d2} and~\eqref{eq:gausssum},    we have the formula 
\begin{gather*}
e\pmfrac{1-d}8(-1)^{\frac{Qd-1}{2} \frac{d-1}{2}} G(Q/d)
=d^{-\frac12}\pmfrac2d G(Q),
\end{gather*}
and  from  \eqref{eq:pmtdef} we have
\begin{gather*}
\pmfrac{24\beta-1}d=\pmfrac{24\beta_0-1}d\ \ \text{for}\ \ d\mid Q.
\end{gather*}
Using these facts, replacing $\frac{24(\ell n+t_{d, \beta})-1}{Q/d}$ by $n$, and recalling the definition \eqref{eq:lambdadef}, we find that 
\begin{gather*}
 F_{\ell,Q,\beta,d}=
 G(Q) \pmfrac{-24 (24\beta_0-1)}{Q}\lambda_d
 \sum_{\frac{Q}{d}n\equiv 24t_{d, \beta} -1\spmod \ell}  p\pmfrac{\frac{Q}{d}n+1}{24} q^{\frac{nd}{24 \ell}},
\end{gather*}
which may be written in the form
\begin{gather}\label{eq:fbetafourier}
 F_{\ell,Q,\beta,d}\big|V_\ell 
=      G(Q) \pmfrac{-24 (24\beta_0-1)}{Q}   \lambda_d \left( \sum_{n \equiv 24 t_{d, \beta} -1 \spmod\ell} p  \pmfrac{  n + 1}{24}  q^{\frac{n}{24}} \right) \big| U_{Q/d} V_{d}.
\end{gather}

We are now in a position to prove Theorem~\ref{thm:theorem1full}.
Suppose first that   $\delta=\left( \frac{1-24\beta_0}{\ell} \right) =0$.
Then for every $\beta\in S_{\ell Q, \beta_0}$ we have $\ell\mid 24\beta-1$.
It follows from \eqref{eq:tddef} that for each $d$ and $\beta$  we have $t_{d, \beta}\equiv \beta_0\pmod\ell$, from which 
\begin{gather*}
 \sum_{n \equiv 24 t_{d, \beta} -1 \spmod\ell} p  \pmfrac{n + 1}{24} q^\frac n{24} = \sum_{\ell\mid n} p  \pmfrac{  n + 1}{24} q^\frac n{24}\equiv f_{\ell,0}\pmod\ell.		
\end{gather*}
Using~\eqref{eq:fbetaddef} and~\eqref{eq:fbetafourier} together with the fact that $F_{ \ell, Q, \beta} \equiv 0 \pmod{\ell}$, we find that 
\begin{gather*}
 \sum_{d \mid Q} \lambda_d  f_{\ell,0} \big|  U_{Q/d}  V_d \equiv  0 \pmod{\ell}\tx{. }
\end{gather*}
Factoring gives
\begin{gather*}
f_{\ell,0} \big|  \left( U_{Q_1} + \lambda_{Q_1}V_{Q_1} \right) \cdots \left(   U_{Q_t} + \lambda_{Q_t}V_{Q_t} \right) \equiv 0 \pmod{\ell}\tx{.}
\end{gather*}
So the theorem follows in this case.

Finally, suppose that~$\delta=\left(\frac{1-24\beta_0}{\ell} \right)= -1$.  In this case, the situation is complicated by the fact that the values of $t_{d, \beta}$ in \eqref{eq:fbetafourier} vary with $d$ and $\beta$. 
 To proceed, define 
 \begin{gather*}
 S':=\{\beta\in S_{\ell Q, \beta_0}: \beta\equiv \beta_0\pmod {Q}\}.
 \end{gather*}
Then $S'$ contains one representative 
for each residue class $\beta\pmod \ell$
with 
\begin{gather*}
\pmfrac{1-24\beta}\ell=\pmfrac{1-24\beta_0}\ell=-1.
\end{gather*}
For each $d$, we see by  \eqref{eq:tddef}     that $t_{d,\beta}$ ranges over those residue classes $t\pmod\ell$ 
with $\pfrac{1-24t}\ell=-1$ as $\beta$ ranges over
$S'$.
We conclude that  
\begin{gather*}
  f_{\ell,-1}
\equiv
  \sum_{\pfrac{-n}{\ell} = -1}
  p\pmfrac{n + 1}{24} q^\frac n{24}
\equiv
  \sum_{\beta \in S'}
  \sum_{n \equiv 24 t_{d,\beta} -1 \spmod\ell}
  p\pmfrac{n + 1}{24} q^\frac n{24}\pmod\ell
\tx{.}
\end{gather*}
Combining this with \eqref{eq:fbetaddef} and \eqref{eq:fbetafourier}, we obtain
\begin{gather*}
\sum_{\beta \in S'} F_{ \ell, Q, \beta}\big|V_\ell=
\sum_{ d \mid Q} \sum_{\beta \in S'} F_{\ell,Q,\beta,d}\big|V_\ell  
\equiv   G(Q) \pmfrac{-24 (24\beta_0-1)}{Q}  \sum_{d\mid Q}  \lambda_d f_{\ell,-1}\big| U_{Q/d} V_d\pmod\ell.
\end{gather*}
Since $F_{ \ell, Q, \beta}\equiv 0 \pmod{\ell}$ for each $\beta$ in the sum, we obtain
\begin{gather*}
 \sum_{d\mid  Q} \lambda_d f_{\ell,-1}\big| U_{Q/d} V_d \equiv 0 \pmod{\ell} \tx{.}
\end{gather*}
Factoring gives 
\begin{gather*}
f_{\ell,-1} \big|  \left( U_{Q_1} + \lambda_{Q_1}V_{Q_1} \right) \cdots \left(  U_{Q_t} + \lambda_{Q_t} V_{Q_t} \right) \equiv 0 \pmod{\ell} \tx{.}
\end{gather*}
This proves Theorem~\ref{thm:theorem1full}.
\end{proof}


\section{Proof of Theorem \ref{thm:main}}\label{sec:proofmain}

Theorem~\ref{thm:main} will follow from the next result together with Theorem~\ref{thm:UQVQ}.
\begin{theorem}\label{thm:upvpgen}
Suppose that $\ell\geq 5$ is prime, 
 that $(r, 24)=1$ and that $f=\sum a(n)q^\frac n{24}\in S_k\(1, \nu_\eta^r\)$
has $f\not\equiv 0\pmod\ell$.
Fix $\ep\in \{\pm 1\}$.  Let $S$ be the set of primes $p$ such that 
\begin{gather}\label{eq:upvpgen}
f\big| U_p\equiv c_p f\big| V_p\pmod\ell \ \ \ \ \text{for some $c_p\not\equiv 0\pmod\ell$}
\end{gather}
and
\begin{gather}\label{eq:twistgen}
f\equiv  \sum_{\pfrac np=\ep} a(n)q^\frac n{24} +\sum a(p^2n)q^\frac{p^2 n}{24}\pmod\ell.
\end{gather}
Then one of the following is true.
\begin{enumerate}
\item $S$ has density zero, or
\item we have 
\begin{gather*}
\#\{ n\leq X\ : a(n)\not\equiv 0\pmod \ell\} \ll \sqrt{X} \log X
\end{gather*}
and 
\begin{gather*}
f\big| T_{Q^2}\equiv 0\pmod\ell\ \ \text{for all primes}\ \  Q\equiv -1\pmod\ell.
\end{gather*}
\end{enumerate}
\end{theorem}

Assuming this result for the moment, we prove Theorem~\ref{thm:main}.
\begin{proof}[{Proof of Theorem~\ref{thm:main}}]
Recall the modular forms~$f_{\ell,\delta}$ from~\eqref{eq:fldeldef}. If $\ell \in \{5,7,11\}$ and~$\delta = 0$, we have~$f_{\ell,\delta} \equiv 0 \pmod{\ell}$ and both~\eqref{eq:al0} and~\eqref{eq:tq0} hold trivially. We exclude these cases from further consideration. In all other cases, we have~$f_{\ell, \delta} \not\equiv 0 \pmod{\ell}$ by the work cited after the statement of Theorem~\ref{thm:main}. 
Fix $\delta\in\{0, -1\}$.  For $\ep\in \{\pm1\}$, let   $S_\ep$ be the set of primes $p$ for which there is a congruence
\begin{gather*}
p(\ell p n+\beta)\equiv 0\pmod\ell\ \ \text{with}\ \ \pmfrac{1-24\beta}\ell=\delta,\ \ \pmfrac{24\beta-1}p=\ep.
\end{gather*}
In view of the first assertion of Theorem~\ref{thm:UQVQ}, we see that   $S=S_1\cup S_{-1}$ is the set described in the statement of Theorem~\ref{thm:main}.

Suppose that $p\in S_\ep$.
Then 
\begin{gather*}
\sum_{\pfrac np=\ep} a_{\ell, \delta}(n)q^\frac n{24}\equiv 0\pmod\ell,
\end{gather*}
and by Theorem~\ref{thm:UQVQ} we have 
\begin{gather*}
f_{\ell, \delta}\big| U_p\equiv -\ep \pfrac{-12}pp^{-1} f_{\ell, \delta}\big| V_p\pmod\ell.
\end{gather*}
Using this fact together with  Theorem~\ref{thm:raduclasses} we see that 
\begin{gather*}
f_{\ell, \delta}=\sum a_{\ell, \delta}(n)q^\frac n{24}\equiv \sum_{\pfrac np=-\ep} a_{\ell, \delta}(n)q^\frac n{24} +\sum a_{\ell, \delta}(p^2n)q^\frac{p^2 n}{24}\pmod\ell.
\end{gather*}
Suppose that  $S$ does not have density zero. Then the same is true of $S_\ep$ for some choice of~$\ep$, and Theorem~\ref{thm:main} follows
from applying Theorem~\ref{thm:upvpgen} to $f_{\ell, \delta}$.
\end{proof}

We  turn to the proof of Theorem~\ref{thm:upvpgen}. Let $f$ and $S$ be as in the hypotheses.
 If $Q$ is a prime with $Q\geq 5$  and $Q\neq \ell$ then one of two mutually exclusive things must occur.
Either
\begin{gather}\label{eq:ordQeven} 
\text{there exists an integer $n_Q$ with $\operatorname{ord}_Q(n_Q)$\  odd and $a(n_Q)\not\equiv 0\pmod\ell$,}
\end{gather}
or 
\begin{gather}\label{eq:uqvqpowers} 
  f\big| U_{Q^{2j+1}}\equiv f\big|U_{Q^{2j+2}}V_Q\pmod\ell\ \ \ \text{for all $j\geq 0$}.
\end{gather}
Note that  if $Q\in S$, then \eqref{eq:uqvqpowers} holds by virtue of \eqref{eq:upvpgen}  and induction on  $\mathrm{ord}_Q(n_Q)$.

Suppose that \eqref{eq:ordQeven} holds for infinitely many primes $Q$.  For each such prime, write 
\begin{gather*}
n_Q= Q^{1 + 2 e_Q}n_Q'\ \ \text{with}\ \  Q\nmid n_Q'.
\end{gather*}
Letting $Q_1$ denote the smallest such prime, we may then choose an infinite sequence of such primes $Q_2, Q_3, \dots$ successively with 
\begin{gather}\label{eq:Qjnmid}
Q_j\nmid n_{Q_1} \cdots n_{Q_{j-1}}.
\end{gather}
For each $j$ we have
\begin{gather}\label{eq:nQjprime}
a(Q_j^{1+2e_{Q_j}}n_{Q_j}')\not\equiv 0\pmod\ell.
\end{gather}
For each $p\in S$,  it follows from \eqref{eq:uqvqpowers}  that $\operatorname{ord}_p(n_{Q_j}')$ is even, and  from \eqref{eq:upvpgen} that
\begin{gather*}
  a(p^2 n)\equiv c_p a(n) \pmod\ell\ \ \ \text{for all $n$}.
\end{gather*}
Using these facts, we may remove even powers of each $p\in S$ from each $n_{Q_j}'$, and we may therefore assume in 
\eqref{eq:nQjprime} that for all $j$ and all $p\in S$ we have 
\begin{gather}\label{eq:pnmid}
p\nmid n_{Q_j}'.
\end{gather}

Let $t\geq 1$.  From \eqref{eq:nQjprime}, \eqref{eq:pnmid} and  \eqref{eq:twistgen}, 
we see that 
\begin{gather*}
  \pmfrac{Q_j n_{Q_j}'}p=\ep\ \ \ \text{for  $p\in S$ and $1\leq j\leq t$.}
\end{gather*}
Each of these quadratic relations imposes a residue class restriction on the primes $p$ that may belong to $S$. More specifically, the quadratic relation corresponding to $Q_j$ prohibits from belonging to $S$ a proportion of $1/2$ of the primes not already prohibited by the quadratic relations corresponding to $Q_1, Q_2, \cdots , Q_{j-1}$. 
Therefore 
\begin{gather*}
  \limsup_{X\to\infty}\frac{\#\{p\in S\ : \ p\leq X\}}{X/\log X}\leq 2^{-t}.
\end{gather*}
Since $t$ is arbitrary we conclude that $S$ has density zero if 
\eqref{eq:ordQeven} holds for infinitely many primes $Q$.

From \eqref{eq:ordQeven} and \eqref{eq:uqvqpowers}  we conclude that to 
prove Theorem~\ref{thm:upvpgen} it  suffices to establish the following proposition,
whose proof occupies the remainder of this section.

\begin{proposition}\label{prop:case2}
Suppose that $\ell\geq 5$ is prime,  that $(r, 24)=1$ and that $f=\sum a(n)q^\frac n{24}\in S_k\(1, \nu_\eta^r\)$
has $f\not\equiv 0\pmod\ell$.
Suppose that for all but finitely many primes $Q\geq 5$   we have
\begin{gather}\label{eq:uqvqj}
f\big| U_{Q^{2j+1}}\equiv f\big|U_{Q^{2j+2}}V_Q\pmod\ell\ \ \ \text{for all $j\geq 0$}.
\end{gather}
Then 
\begin{gather*}
\#\{ n\leq X\ : a(n)\not\equiv 0\pmod \ell\} \ll \sqrt{X} \log X,
\end{gather*}
and 
\begin{gather*}
f\big| T_{Q^2}\equiv 0\pmod\ell\ \ \text{for all primes}\ \  Q\equiv -1\pmod\ell.
\end{gather*}
\end{proposition}

In the proof of Proposition~\ref{prop:case2} we will need the following result.

\begin{proposition}\label{prop:uquqsquare}
Let $\ell\geq 5$ and $Q\geq 5$ be primes with  $Q\neq\ell$ and $Q\not\equiv 1\pmod\ell$.  Suppose that $(r, 24)=1$ and that 
$f\in M_k\(1,\nu_\eta^r\)$ has 
\begin{gather}\label{eq:uquqsquare}
f\big| U_Q\equiv f\big|U_{Q^2} V_Q\pmod \ell.
\end{gather}
Then we have 
\begin{gather}\label{eq:fuqsquare}
f\big| U_{Q^2}\equiv Q^{k-\frac12} \pmfrac{-1}Q^{k-\frac12}\pmfrac{12}Q f\otimes \chi_Q+Q^{2k-1}f\big| V_{Q^2}\pmod\ell,
\end{gather}
and 
\begin{gather}\label{eq:ftqsquare}
f\big| T_{Q^2}\equiv \(1+Q^{-1}\) f\big| U_{Q^2}\pmod\ell.
\end{gather}
\end{proposition}

\begin{proof}[Proof of Proposition~\ref{prop:uquqsquare}]
Let $\ell$, $Q$ and $f$ be as in the hypothesis.
By \eqref{eq:uqaction} and \eqref{eq:vqaction} we have 
\begin{gather*}
  f \big| U_Q
\in
  M_{k}\(Q, \chi_Q\nu_\eta^{Qr}\)
\tx{,}\quad
  f \big| U_{Q^2} V_Q
\in
  M_{k}\(Q^2, \chi_Q\nu_\eta^{Qr}\)
\tx{.}
\end{gather*}
Then~\eqref{eq:uquqsquare} and~\eqref{eq:qexp} give
\begin{gather}
\label{eq:uq-qexp}
f\big| U_Q\sk W_Q\equiv f\big|U_{Q^2} V_Q\sk W_Q\pmod \ell.
\end{gather}

We first claim that 
\begin{gather}\label{eq:uqwq}
  f\big|U_Q\sk W_Q
=
  Q^{\frac k2-1}e\pmfrac{-rQ}{8}\pmfrac{-24}Q G(Q)\, f\otimes \chi_Q
  +
  Q^{\frac{3k}2-1}e\pmfrac{-r}8 f\big|V_{Q^2}
.
\end{gather}
To establish \eqref{eq:uqwq} we compute using Lemma~\ref{lem:uQ}.
We have 
\begin{gather}\label{eq:uqwq1}
f\big|U_Q\sk W_Q=Q^{\frac k2-1}\sum_{v\spmod Q}f\sk\left[\pmatrix{24vQ}{-1}{Q^2}0, Q^\frac12z^\frac12\right].
\end{gather}
Using \eqref{eq:nueta}, the term with $v=0$ in \eqref{eq:uqwq1} gives 
\begin{multline}\label{eq:vzero}
Q^{\frac k2-1} f\sk\pmatrix 0{-1}{1}0^*\left[\pmatrix {Q^2}001, Q^{-\frac12}\right]
=Q^{\frac k2-1}e\pmfrac{-r}8f\sk\left[\pmatrix {Q^2}001, Q^{-\frac12}\right]
\\=Q^{\frac{3k}2-1}e\pmfrac{-r}8 f\big|V_{Q^2}.
\end{multline}
For each  $v\not\equiv 0\pmod Q$ in \eqref{eq:uqwq1} we choose $\alpha$ with 
\begin{gather*}
24 \alpha v\equiv 1\pmod Q\ \ \ \ \text{and} \ \  \alpha=24 \alpha'\ \ \text{with}\ \  \alpha'\in \Z.
\end{gather*}
Then the terms in \eqref{eq:uqwq1} with $v\not\equiv 0\pmod Q$ contribute
\begin{gather*}
Q^{\frac k2-1}\sum_{v\spmod Q^*}f\sk \pmatrix{24v}{\frac{24\alpha v-1}Q}Q\alpha^*\left[ \pmatrix Q{-\alpha}0Q, 1\right].
\end{gather*}
Using \eqref{eq:nueta} this becomes 
\begin{gather*}Q^{\frac k2-1}e\pmfrac{-rQ}8  \sum_{v\spmod Q^*}\pmfrac\alpha Q f\(\tau-\mfrac\alpha Q\)=
Q^{\frac k2-1}e\pmfrac{-rQ}8 \sum_n a(n)q^\frac n{24}  \sum_{v\spmod Q^*}\pmfrac{24\alpha'} Q \zeta_Q^{-n\alpha'}.
\end{gather*}
Since the inner sum evaluates to $\pfrac{-24 n}Q G(Q)$, the claim \eqref{eq:uqwq} follows from this equation together with \eqref{eq:vzero}.

We next claim that if \eqref{eq:uquqsquare} holds then 
\begin{gather}\label{eq:uqwq2}
f\big| U_{Q^2}V_Q\sk W_Q\equiv Q^{-1}f\big| U_Q\sk W_Q+Q^{-\frac k2}(1-Q^{-1})e\pmfrac{-r}8 f\big|U_{Q^2}\pmod\ell.
\end{gather}
To prove this, we use Lemma~\ref{lem:uQ}  to compute 
\begin{gather}\label{eq:uqvqwq}
f\big| U_{Q^2}V_Q\sk W_Q=Q^{\frac k2-2}\sum_{v\spmod{Q^2}} f\sk \left[\pmatrix{24v}{-1}{Q^2}0, Q^\frac12z^\frac12\right].
\end{gather}
Using \eqref{eq:uqwq1}, we see that the terms in \eqref{eq:uqvqwq} with $v\equiv 0\pmod Q$ give
\begin{gather}\label{eq:vzero1}
Q^{\frac k2-2}\sum_{v'\spmod Q}f\sk\left[\pmatrix{24v'Q}{-1}{Q^2}0, Q^\frac12z^\frac12\right]=Q^{-1}f\big|U_Q\sk W_Q.
\end{gather}
For $v\not\equiv0\pmod Q$,  choose $\alpha$ with 
\begin{gather*}
24 \alpha v\equiv 1\pmod {Q^2}\ \ \ \ \text{and} \ \  \alpha=24 \alpha'\ \ \text{with}\ \  \alpha'\in \Z.
\end{gather*}
Using \eqref{eq:nueta}, the terms in \eqref{eq:uqvqwq} with $v\not\equiv0\pmod Q$ give
\begin{multline}\label{eq:vnonzero}
  Q^{\frac k2-2}
  \sum_{\substack{v \spmod{Q^2}\\v\not\equiv 0\spmod Q}}
  f\sk\pmatrix{24v}{\frac{24\alpha v-1}{Q^2}}{Q^2}\alpha^*\left[\pmatrix1{-\alpha}0{Q^2}, Q^\frac12\right]
=
  Q^{\frac k2-2}e\pmfrac{-r}8
  \sum_{\substack{v \spmod{Q^2}\\v\not\equiv 0\spmod Q}}
  f\sk\left[\pmatrix1{-\alpha}0{Q^2}, Q^\frac12\right]\\
=
  Q^{-\frac k2-2}e\pmfrac{-r}8
  \sum_n a(n)q^\frac n{24Q^2}
  \sum_{\substack{v \spmod{Q^2}\\v\not\equiv 0\spmod Q}}
  \zeta_{Q^2}^{-n\alpha'}
\text{.}
\end{multline}
The inner sum is
\begin{gather*}
  \sum_{\substack{v\spmod{Q^2}\\v\not\equiv 0\spmod Q}}
  \zeta_{Q^2}^{n v}
=
  \sum_{v\spmod{Q^2}}\zeta_{Q^2}^{nv}-\sum_{v\spmod Q}\zeta_{Q}^{nv}
=\begin{cases}
	Q^2-Q\ \ &\text{if $Q^2\mid n$},\\
	-Q \ \ &\text{if $Q\mid\mid n$},\\
	0 \ \ &\text{if $Q\nmid n$}.\\
\end{cases}
\end{gather*}
From \eqref{eq:uquqsquare} we have $a(n)\equiv 0\pmod\ell$ if $Q\mid\mid n$.  Therefore the expression in 
\eqref{eq:vnonzero} is congruent to 
\begin{gather*}
  Q^{-\frac k2-2}(Q^2-Q)e\pmfrac{-r}8\sum a(Q^2n)q^\frac{n}{24}
\equiv   Q^{-\frac k2}(1-Q^{-1})e\pmfrac{-r}8f\big|U_{Q^2} \pmod\ell.
\end{gather*}
The claim \eqref{eq:uqwq2} follows from this together with \eqref{eq:vzero1}.

If $Q\not\equiv 1\pmod\ell$ it follows from \eqref{eq:uq-qexp} and  \eqref{eq:uqwq2} that 
\begin{gather*}
f\big| U_Q\sk W_Q\equiv Q^{-\frac k2}e\pmfrac{-r}8 f\big|U_{Q^2} \pmod \ell.
\end{gather*}
Combining this with  \eqref{eq:uqwq} gives
\begin{gather*}
f\big| U_{Q^2}\equiv Q^{k-1}e\pmfrac{r(1-Q)}8\pmfrac{-24}Q G(Q)\, f\otimes \chi_Q+Q^{2k-1} f \big| V_{Q^2}\pmod\ell
\tx{.}
\end{gather*}
 From \eqref{eq:oneminusd-epd1d2} and~\eqref{eq:gausssum}, we obtain
\begin{gather*}
f\big| U_{Q^2}\equiv Q^{k-\frac12}\ep_Q^{r+1}\pmfrac{-12}Q  \, f\otimes \chi_Q+Q^{2k-1} f \big| V_{Q^2}\pmod\ell 
\tx{.}
\end{gather*}
From \eqref{eq:krcond} we have $2k\equiv r\pmod 4$.
This gives \eqref{eq:fuqsquare},  and  \eqref{eq:ftqsquare} then  follows from the definition~\eqref{eq:heckedef} of the Hecke operator.
\end{proof}

We require a lemma before turning to the proof of Proposition~\ref{prop:case2}.

\begin{lemma}\label{lem:chebotarev}  Under the hypotheses of Proposition~\ref{prop:case2}, 
suppose that $Q_0\geq 5$, $Q_0\neq\ell$ is a prime with $Q_0\not\equiv \pm 1\pmod\ell$ for which \eqref{eq:uqvqj} holds.
Define
\begin{multline*}
  \mQ
:=
  \big\{
  Q\text{ prime}
  \,:\quad
  \text{\eqref{eq:uqvqj} holds for $Q$,}\quad
  Q \equiv Q_0\pmod{12\ell}
  \quad\text{and}
\\
  f\big| T_{Q^2}\equiv f\big|T_{Q_0^2}\pmod\ell
  \big\}
\tx{.}
\end{multline*}
Then 
\begin{gather}\label{eq:mqden}
\#\{ Q\in \mQ\ : \ Q\leq X\} \gg \frac X{\log X}.
\end{gather}
\end{lemma}
\begin{proof}[Proof of Lemma~\ref{lem:chebotarev}]

Suppose that $f\in S_k\(1, \nu_\eta^r\)$ and that $f\not\equiv 0\pmod\ell$.  
For each squarefree $t$ let
\begin{gather*}
F_t\in S_{2k-1}^{\operatorname{new}}(6)
\end{gather*}
be the form with 
\begin{gather*}
\Sh_t f=F_t\otimes\pmfrac{12}\bullet.
\end{gather*}

By \eqref{eq:shimcong} and \eqref{eq:shimliftcommute}, 
for all primes $Q\geq 5$ we have 
\begin{gather}\label{eq:ftzero}
f\big|T_{Q^2}\equiv f\big|T_{Q_0^2} \pmod\ell \iff F_t\big|T_Q\equiv F_t\big|T_{Q_0}\pmod\ell\ \ \ \text{for all $t$}.
\end{gather}
As $t$ ranges over all squarefree integers, there are only finitely many non-zero possibilities for $F_t\pmod\ell$.
Let  $\{F_{t_1}, \dots, F_{t_k}\}$ be a collection which represents all of these possibilities; from the definition  \eqref{eq:shimliftcoeff} we see that this collection is
not empty.

The space 
$S_{2k-1}^{\operatorname{new}}(6 )$ is spanned by newforms $g_1, \dots, g_d$. 
Write
\begin{gather*}
F_{t_j}=\sum_{i=1}^d c_{i, j} g_i.
\end{gather*}
 Let $L$ be the number field generated by the coefficients of $g_1, \dots g_d$ as well as the collection $\{c_{i, j}\}$
  and let $\mathcal{O}_L$ be the ring of integers.
 Let $\pi$ be a prime ideal above $\ell$ in $\mathcal{O}_L$.
Define 
 \begin{gather*}
   m := \max(1, 1 - \min(\ord_{\pi}(c_i))) > 0.
 \end{gather*}
 For each  $i$, it follows from the work of Deligne  (see, for example, \cite[Thm. 6.7]{Deligne-Serre})   that there is
a Galois representation 
\begin{gather*}
\rho_i : \Gal(\overline{\Q}/\Q) \rightarrow \GL_2(\mathcal{O}_L/\pi^m),
\end{gather*}
 unramified outside of $6\ell$, such that if
\begin{gather*}
g_i =\sum_{n=1}^{\infty} a_i(n)q^n,
\end{gather*} then for all primes  $Q \nmid 6\ell$, we have
\begin{gather*}
\Tr  (\rho_i(\Frob_Q)) \equiv a_i(Q) \pmod{\pi^m}.
\end{gather*}
Let $\psi$ denote the mod $12\ell$ cyclotomic character.  

All of the $\rho_i$ and~$\psi$ have finite images.
By the Chebotarev density theorem, there is a set of primes $\mathcal P$ with  positive lower density such that for  $Q\in \mathcal P$ we have 
\begin{gather*}
a_i(Q)\equiv a_i(Q_0)\pmod{\pi^m}\ \ \text{for all $i$}
\end{gather*} 
and $\psi(Q)=\psi(Q_0)$.  
For such $Q$ we have
\begin{gather*}
F_{t_j}\big|T_Q=\sum_{i=1}^d c_{i, j} a_i(Q)g_i\equiv \sum_{i=1}^d c_{i, j} a_i(Q_0)g_i\equiv F_{t_j}\big|T_{Q_0}\pmod\pi.
\end{gather*}
By  assumption, all but finitely many primes~$Q$ satisfy \eqref{eq:uqvqj}. The lemma follows from \eqref{eq:ftzero}.
\end{proof}

\begin{proof}[Proof of Proposition~\ref{prop:case2}]
Fix a prime $Q_0\geq 5$ with $Q_0\neq \ell$ and  $Q_0\not\equiv \pm 1\pmod\ell$ for which \eqref{eq:uqvqj}  holds,
and let $\mQ$ be the set provided by Lemma~\ref{lem:chebotarev} (recall that $Q_0\in \mathcal Q$).
Define
\begin{gather*}
c:=Q_0^{k-\frac12} \pmfrac{-1}{Q_0}^{k-\frac12}\pmfrac{12}{Q_0}.
\end{gather*}
Then for $Q\in \mQ$, \eqref{eq:fuqsquare} and \eqref{eq:ftqsquare} give 
\begin{gather}\label{eq:QQzero}
\sum \pmfrac nQ a(n)q^\frac n{24}+c \sum a(n) q^\frac{Q^2 n}{24}\equiv \sum \pmfrac n{Q_0} a(n)q^\frac n{24}+c \sum a(n) q^\frac{Q_0^2 n}{24}\pmod\ell.
\end{gather}
It follows that
\begin{gather}\label{eq:QQzero1}
\sum_{(n, QQ_0)=1} \left[\pmfrac nQ-\pmfrac n{Q_0}\right] a(n)q^\frac n{24}\equiv 0\pmod\ell.
\end{gather}

From  \eqref{eq:uqvqj}, we see that 
\begin{gather*}
a(n)\not\equiv0\pmod\ell\implies \operatorname{ord}_Q(n)\ \ \text{is even for all $Q\in \mQ$},
\end{gather*}
  and from \eqref{eq:fuqsquare} we see that for all $Q\in \mQ$ we have
\begin{gather*}
  a(Q^2n)\not\equiv 0\pmod\ell \implies a(n)\not\equiv 0\pmod\ell.
\end{gather*}
We conclude that if $a(n)\not\equiv 0\pmod\ell$ then 
\begin{gather}\label{eq:Mdef}
\text{$n=M^2 n'$ where  $M$ is divisible only by primes in $\mQ$, and $Q\nmid n'$ for all $Q \in \mQ$.}
\end{gather}
Moreover,  for such values of $n$, \eqref{eq:QQzero1} gives 
\begin{gather}\label{eq:quadcond}
\pmfrac{n'}Q=\pmfrac{n'}{Q_0}\ \ \ \text{for all $Q\in \mQ$}.
\end{gather}

We now apply the arithmetic large sieve \cite{Montgomery}.  Let $N$ be a parameter, and let $\mN$ be the set of integers 
consisting of all $n\leq N$ with $a(n)\not\equiv 0\pmod\ell$. Write $\mN = \mN_1 \cup \mN_{-1}$, where
\begin{gather}
\mN_{\ep} := \left\{ n \in \mN: \pmfrac{n'}{Q_0} = \ep \right\},
\end{gather}
and let $Z_{\ep} := | \mN_{\ep }|$. 

Fix $\ep \in \{ \pm 1\}$ and write $\mN_{\ep} = \{ n_1, \cdots, n_{Z_{\ep}} \}$.   Using \eqref{eq:Mdef} and \eqref{eq:quadcond} we see that for each $i\in \{1, \dots, Z_{\ep}\}$, we have
\begin{gather*}\pmfrac{n_i}Q= 0\ \ \ \text{or}\ \ \ \pmfrac{n_i}Q= \ep \ \ \ \text{for all $Q\in \mQ$}.\end{gather*}
For each prime $Q$,  define 
\begin{gather*}w(Q):=\begin{cases}
	\frac{Q-1}2\  \ &\text{if $Q\in \mQ$},\\
	0\ \  &\text{else}.
	\end{cases}
\end{gather*}
Let $X$ be another parameter.
Then for each $Q\leq X$, there are $w(Q)$  residue classes modulo $Q$ which  contain no element of $\mN_{\ep}$.
By \cite{Montgomery} we have
\begin{gather*} 
Z_{\ep} \leq \mfrac{\(N^{\frac12}+X\)^2}M,
\end{gather*}
with 
\begin{gather*}M:=\sum_{\substack{Q\leq X\\ Q\ \text{prime}}}\frac{w(Q)}{Q-w(Q)}=
\sum_{\substack{Q\leq X\\Q\in \mQ} }\frac{Q-1}{Q+1}\gg \frac X{\log X},\end{gather*}
where in the last estimate we have used \eqref{eq:mqden}.
Setting $X=\sqrt{N}$ gives $Z_{\ep} \ll \sqrt N\log N$.  Thus 
\begin{gather*}
|\mN| = Z_1 + Z_{-1} \ll \sqrt N \log N.
\end{gather*}
This proves the first assertion in Proposition~\ref{prop:case2}.

We turn to the second assertion of Proposition~\ref{prop:case2}.    By the hypotheses,  for  all but finitely many primes  $Q\equiv -1\pmod\ell$,
\eqref{eq:ftqsquare} gives
\begin{gather}\label{eq:tqzero}
f\big| T_{Q^2}\equiv 0\pmod \ell.
\end{gather}
Suppose that there is a prime $Q_0\equiv -1\pmod\ell$ for which 
$f\big| T_{Q_0^2}\not\equiv 0\pmod \ell$.
Arguing as in Lemma~\ref{lem:chebotarev}, we see that a positive proportion of primes $Q\equiv -1\pmod\ell$ have
\begin{gather*}
f\big| T_{Q^2}\equiv f\big| T_{Q_0^2}\not\equiv 0\pmod\ell.
 \end{gather*}
 It follows that  such a prime $Q_0$ does not exist.  In other words, \eqref{eq:tqzero} holds for all $Q\equiv -1\pmod\ell$.
The last assertion of  Proposition~\ref{prop:case2} follows.
\end{proof}


\section{Proof of Theorem~\ref{thm:UQQVQQ}}\label{sec:square}

We start by proving a lemma which we will employ to handle the first  case of Theorem~\ref{thm:UQQVQQ}. The proof is similar in spirit to the proof of Lemma~\ref{lem:uq}.
Since there are a number of technical differences  we  present a self-contained proof for the reader's benefit.
\begin{lemma}\label{lem:levelQuQ}
Suppose that $\ell \geq 5$ and $Q \geq 5$ are primes with $Q \neq \ell$.  
Suppose that $(r, 24)=1$, that $a\in \{0, 1\}$, and that $f\in M_k\(Q,\chi_Q^a\nu_\eta^r\)$.
Then
\begin{gather*}
  f \big| U_Q \equiv 0 \pmod\ell
\iff
  f \equiv 0 \pmod\ell
\tx{.}
\end{gather*}
\end{lemma}
\begin{proof}
Only one direction requires proof.
By Lemma~\ref{lem:uQ}, we have 
\begin{gather}\label{eq:levq1}
f\big|U_Q\sk W_Q=Q^{\frac k2-1}\sum_{v \spmod{Q}}f\sk\left[\pmatrix{24vQ}{-1}{Q^2}0, Q^\frac12z^\frac12\right].
\end{gather}
The term with $v=0$ in \eqref{eq:levq1} gives 
\begin{gather}\label{eq:levq2}
Q^{\frac k2-1}f\sk\left[\pmatrix{0}{-1}Q0, Q^\frac14z^\frac12\right]\left[\pmatrix Q001, Q^{-\frac14}\right]=Q^{k-1}f\sk W_Q\big|V_Q.
\end{gather}
For each  $v\not\equiv 0\pmod Q$ in \eqref{eq:levq1} we choose $\alpha$ with 
\begin{gather*}
24 \alpha v\equiv 1\pmod Q\ \ \ \ \text{and} \ \  \alpha=24 \alpha'\ \ \text{with}\ \  \alpha'\in \Z.
\end{gather*}

Write $f=\sum a(n)q^\frac n{24}$.
From  \eqref{eq:nueta},  the terms in \eqref{eq:levq1} with $v\not\equiv 0\pmod Q$ contribute
\begin{multline}\label{eq:levq3}
Q^{\frac k2-1}\sum_{v \spmod{Q}^*}f\sk \pmatrix{24v}{\frac{24\alpha v-1}Q}Q\alpha^*\left[ \pmatrix Q{-\alpha}0Q, 1\right]
=Q^{\frac k2-1}e\pmfrac{-rQ}8\sum_{v \spmod{Q}^*}\pmfrac\alpha Q^{a+r}f\sk \left[\pmatrix Q{-\alpha}0Q, 1\right]\\
=Q^{\frac k2-1}e\pmfrac{-rQ}8\sum_n a(n)q^\frac n{24} \sum_{v \spmod{Q}^*}\pmfrac\alpha Q^{a+r}\zeta_Q^{-n\alpha'}.
\end{multline}
Suppose first that $a+r$ is even.
Then the inner sum is $Q-1$ if $Q\mid n$ and $-1$ if $Q\nmid n$.
So the last expression becomes
\begin{gather*}
Q^{\frac k2-1}e\pmfrac{-rQ}8(Q-1)f\big|U_Q V_Q-Q^{\frac k2-1}e\pmfrac{-rQ}8\sum_{Q\nmid n}a(n)q^\frac n{24}.
\end{gather*}
From this and \eqref{eq:levq2} we obtain
\begin{gather*}
f\big|U_Q\sk W_Q=Q^{k-1}f\sk W_Q\big|V_Q
+Q^{\frac k2-1}e\pmfrac{-rQ}8(Q-1)f\big|U_Q V_Q-Q^{\frac k2-1}e\pmfrac{-rQ}8\sum_{Q\nmid n}a(n)q^\frac n{24}.\end{gather*}
If $f\big|U_Q\equiv 0\pmod\ell$ then $f\big|U_QV_Q \equiv 0\pmod\ell$, and
$f\big|U_Q\sk W_Q\equiv 0\pmod\ell$ by~\eqref{eq:qexp}.
Therefore 
\begin{gather}\label{eq:levq31}
Q^{k-1}f\sk W_Q\big|V_Q\equiv Q^{\frac k2-1}e\pmfrac{-rQ}8\sum_{Q\nmid n}a(n)q^\frac n{24}\pmod\ell.
\end{gather}

We have  
\begin{gather*}
W_Q\pmatrix1{24}01^*=\pmatrix10{-24Q}1^*W_Q=\left[\pmatrix{-1}00{-1}, -i\right]\pmatrix{-1}0{24Q}{-1}^*W_Q.
\end{gather*}
From  \eqref{eq:nueta} we see that $(f\sk W_Q)(\tau+24)=(-i)^{-2k}e\pfrac{-r}4\chi_Q^a(-1)(f\sk W_Q)(\tau)$.
The compatibility condition \eqref{eq:krcond} can be expressed as $\chi_Q^a(-1)=\pfrac{-1}r(-1)^{k-\frac12}$;
it follows that $(f\sk W_Q)(\tau+24)=(f\sk W_Q)(\tau)$.
 Therefore the  Fourier expansion of the left side of \eqref{eq:levq31} is supported on exponents in $\frac{Q}{24}\mathbb{Z}$. 
 Since the right side is supported on  exponents prime to  $Q$, we see that $f\sk W_Q\equiv 0\pmod\ell$.
Using \eqref{eq:qexp} again, we obtain $f\equiv 0\pmod\ell$.

If on the other hand $a+r$ is odd,  the inner sum in \eqref{eq:levq3} reduces to $\pfrac{-24 n}Q G(Q)$, and we obtain
\begin{gather*}
f\big|U_Q\sk W_Q=Q^{k-1}f\sk W_Q\big|V_Q+
Q^{\frac k2-1}e\pmfrac{-rQ}8\pmfrac{-24}Q G(Q) f\otimes \chi_Q.
\end{gather*}
If $f\big|U_Q\equiv 0\pmod\ell$ then we conclude that $f\equiv 0\pmod\ell$ as before.
\end{proof}

We will need the following statement about Kloosterman sums
\begin{gather*}
  K(a,b,c)
:=
  \sum_{n \,(c)^\ast}
  e\left(\mfrac{a n + b \overline{n}}{c}\right)
\tx{,}
\end{gather*}
where the summation is over residue classes prime to~$c$ and~$n \overline{n} \equiv 1 \;\pmod{c}$.
\begin{lemma}
\label{la:kloosterman_sum_nonzero}
Let~$a, b$ be integers and~$p$, $\ell$ distinct primes, $\ell$~odd. Then we have
\begin{gather*}
  K(a',b',p) \not\equiv 0 \;\pmod{\ell}
\tx{.}
\end{gather*}
\end{lemma}
\begin{remark}
The authors are grateful to Will Sawin, who pointed out a proof of this lemma.
\end{remark}
\begin{proof}
The Kloosterman sum takes values in algebraic integers of the~$p$-th cyclotomic field:
\begin{gather*}
  \Z(\zeta_p)
\cong
  \Z[X] \slash \big( X^{p-1} + X^{p-2} + \cdots + 1 \big)
\cong
  \big( \Z[X] \slash (X^p - 1) \big) \big\slash \big( X^{p-1} + X^{p-2} + \cdots + 1 \big)
\tx{.}
\end{gather*}
Reduction modulo~$\ell$ intertwines with the polynomial quotient, yielding a ring over the finite field with~$\ell$ elements. It therefore suffices to examine the sum
\begin{gather*}
  \sum_{0 < s < p}
  X^{as + b\overline{s}}
=
  \sum_{i = 0}^{p-1}
  c_i X^i
\in
  \Z[X] \slash (X^p - 1)	
\tx{.}
\end{gather*}
The lemma follows if we show that the coefficients~$c_i$ of~$X^i$, $0 \le i < p$, are not all equal modulo~$\ell$. We suppose the opposite and derive a contradiction.

Notice that~$as + b \overline{s} \equiv h \,\pmod{p}$ for any~$h \,\pmod{p}$ yields a qua\-dra\-tic equation in~$s$, which has at most two solutions. In particular, we have~$c_i \in \{0, 1, 2\}$. Since~$\ell$ is odd, $0$, $1$, and~$2$ are distinct modulo~$\ell$. We are assuming that all~$c_i$ are equal modulo~$\ell$, and hence they coincide as integers. Comparing the number of terms in the sum with the possible values of the~$c_i$, we find that $p-1 = \sum_i c_i = p c_0 \in \{0, p, 2p\}$, a contradiction.
\end{proof}

\begin{proof}[{Proof of Theorem~\ref{thm:UQQVQQ}}]
 
We begin with the first assertion. Suppose that there is a congruence \eqref{eq:quadcong} with $Q^2 \mid 24\beta_{\ell, Q}-1$, 
and set $\delta=\pfrac{1-24\beta_{\ell, Q}}\ell$.  Then  Theorem~\ref{thm:raduclasses} implies that 
$f_{\ell,\delta} \big| U_{Q^2} \equiv 0 \pmod{\ell}$, while the first assertion of Theorem~\ref{thm:UQVQ}  implies that $f_{\ell,\delta} \big| U_Q \not\equiv 0 \pmod{\ell}$.
In view of \eqref{eq:uqaction}, this contradicts Lemma~\ref{lem:levelQuQ}.

The second assertion  follows directly  from Lemma 4.5 of \cite{Radu_aoconj}.
The proof of the third assertion  is more difficult.
To begin, we proceed as in the proof of Theorem~\ref{thm:UQVQ}.   Suppose that there is a congruence
\begin{gather*}
p(\ell Q^2 n + \beta_0) \equiv 0 \pmod\ell \ \ \ \text{with}  \ \ \  Q\mid\mid 24\beta_0-1,
\end{gather*}
 and define 
\begin{gather*}
\delta:=\pmfrac{1-24 \beta_0}{\ell}\in\{0, -1\}.
\end{gather*}
Recalling the definitions
\eqref{eq:pmtdef}, \eqref{eq:gammadef},
suppose that
\begin{gather*}
\beta \in S_{\ell Q^2, \beta_0},
\end{gather*}
and let 
\begin{gather*}
  \gamma_{\ell, Q^2}
=
  \pmatrix{1}{-24^2\ell X}{\ell}{Q^2Y}\in \SL_2({\Z})
\tx{.}
\end{gather*}
Recall the definition \eqref{eq:gmdef}.
Arguing as in \eqref{eq:giszero} with $Q$ replaced by $Q^2$, we conclude that 
\begin{gather*}
(\ell\tau+Q^2Y)^\frac12g(\ell Q^2, \beta,  \gamma_{\ell,Q^2} \tau)\equiv 0\pmod\ell.
\end{gather*}

  Let $H_{\ell, Q^2, \beta}$ denote the resulting right side of  \eqref{eq:Radutransformation} with $Q$ replaced by $Q^2$; we have 
\begin{gather}\label{eq:fqsqzero}
H_{\ell, Q^2, \beta}\equiv 0\pmod\ell.
\end{gather}
There are three terms in the sum defining $H_{\ell, Q^2, \beta}$.
Recalling that $T(n, Q^2)=1$, the $d=Q^2$ term becomes
\begin{gather}\label{eq:qsqterm}
Q^{-1} q^{\frac{\(24 t_{Q^2,\beta} - 1\)Q^2}{24 \ell}}
  \sum_{n=0}^{\infty} p(\ell n+ t_{Q^2,\beta} ) q^{ Q^2 n}
=
Q^{-1}
  \sum_{n \equiv 24 t_{Q^2,\beta} - 1 \spmod{ \ell}}
  p\pmfrac{n+1}{24} q^{\frac{Q^2 n}{24 \ell}}
\tx{.}
\end{gather}

For the $d=Q$ term,  \eqref{eq:tndsalie}    gives
\begin{gather*}
  T(n,Q)
=
  \pmfrac{-\ell}{Q}
  S\big(24 \beta -1, \overline{24}^2 \overline{\ell}^4 (24(\ell n + t_{Q,\beta}) -1), Q \big)
\tx{.}
\end{gather*}
Since $Q \mid   24\beta-1$, this reduces to 
\begin{gather*}
  T(n,Q)
=
  \pmfrac{-\ell}{Q}
  \pmfrac{24(\ell n + t_{Q,\beta}) - 1}{Q}
  G(Q)
\tx{.}
\end{gather*}
Substituting~$n$ for $24 (\ell n + t_{Q,\beta}) - 1$ and using~\eqref{eq:oneminusd-epd1d2} and~\eqref{eq:gausssum},  the $d=Q$~term 
 becomes
\begin{multline}\label{eq:qterm}
Q^{-\frac12}e\pmfrac{1-Q}8 G(Q) (-1)^\frac{Q-1}2  \pmfrac{-24}{Q} 
  \sum_{n \equiv 24 t_{Q,\beta} -1 \spmod{\ell}}
  \pmfrac{n}{Q} p\pmfrac{n+1}{24} q^{\frac{n}{24\ell}}
\\
=
  \pmfrac{-12}{Q}
  \sum_{n \equiv 24 t_{Q,\beta} -1 \spmod{\ell}}
  \pmfrac{n}{Q} p\pmfrac{n+1}{24} q^{\frac{n}{24\ell}}
\tx{.}
\end{multline}

For the $d=1$ term, \eqref{eq:tndsalie} gives
\begin{gather*}
  T(n,1)
=
  S\big( 24 \beta -1, \overline{24}^2 \overline{\ell}^4 (24(\ell n + t_{1,\beta}) -1), Q^2 \big)
\tx{.}
\end{gather*}
 This Salie sum vanishes whenever $Q \nmid 24 (\ell n + t_{1,\beta}) -1$ (this follows from the explicit evaluation \cite[Lemma~12.4]{iwaniec_kowalski}).
For the other terms, the  Salie sum reduces to a Kloosterman sum; we have 
\begin{gather*}
  T(n,1)
=
  Q\, K \left( \mfrac{24\beta-1}{Q}, \overline{24}^2 \overline{\ell}^4 \mfrac{24(\ell n + t_{1,\beta}) - 1}{Q} , Q \right)\ \ \ \text{if}\ \ \ Q \mid 24 (\ell n + t_{1,\beta}) -1
\tx{.}
\end{gather*}
When  $Q^2\mid 24(\ell n + t_{1,\beta}) - 1$
the Kloosterman sum reduces to  a Ramanujan sum which evaluates to $-1$. 
 Changing variables as above,
the $d=1$~term of the sum defining $H_{\ell, Q^2, \beta}$ becomes
\begin{multline}\label{eq:qzeroterm}
  - Q \sum_{Q^2 n \equiv 24 t_{1, \beta}-1 \spmod{ \ell}}
  p\pmfrac{Q^2 n +1}{24} q^{\frac{n}{24\ell}}
\\
  +
 Q  \sum_{\substack{Qn \equiv 24t_{1, \beta} -1 \spmod{ \ell} \\ Q \nmid n}}
  K \left( \mfrac{24\beta-1}{Q}, \overline{24}^2 \overline{\ell}^4n , Q \right)
  p\pmfrac{Qn+1}{24} q^{\frac{n}{24 \ell Q}}
\tx{.}
\end{multline}

Using  \eqref{eq:qsqterm}, \eqref{eq:qterm} and \eqref{eq:qzeroterm}, we collect the terms of $H_{\ell, Q^2, \beta}$ whose exponents are in  $\frac1{24\ell}\Z$   to find that 
the following expression vanishes modulo~$\ell$:
\begin{multline}
\label{eq:integralpowers}
 Q^{-1}  \left(
 \sum_{n \equiv 24 t_{Q^2,\beta} - 1 \spmod \ell}
  p\pmfrac{n+1}{24} q^{\frac{n}{24\ell}}
  \right) \Big| V_{Q^2}
  \,+\,
  \pmfrac{-12}{Q}
  \sum_{n \equiv 24 t_{Q,\beta} - 1 \spmod{ \ell}}
 \pmfrac n Q
  p\pmfrac{n+1}{24}
  q^{\frac{n}{24\ell}}
\\
  -\,
  Q
  \left(
  \sum_{n \equiv 24 t_{1, \beta} - 1 \spmod{\ell}}
  p\pmfrac{n+1}{24}
  q^{\frac{n}{24\ell}}
  \right) \Big| U_{Q^2}
\tx{.}
\end{multline}
On the other hand, the sum of   those terms of $H_{\ell, Q^2, \beta}$ whose exponents are not in  $\frac1{24\ell}\Z$
must also vanish modulo $\ell$, giving 
\begin{gather}\label{eq:nonintegralpowers}
  \sum_{\substack{Qn \equiv 24t_{1, \beta} -1 \spmod{\ell} \\ Q \nmid n}}
  K \left( \mfrac{24\beta-1}{Q}, \overline{24}^2 \overline{\ell}^4n , Q \right)
  p\pmfrac{Qn+1}{24\ell}
  q^{\frac{n}{24 \ell Q}}
  \equiv
  0
  \pmod{\ell}
\tx{.}
\end{gather}

Define 
 \begin{gather*}
 S':=\{\beta\in S_{\ell Q^2, \beta_0}: \beta\equiv \beta_0\pmod {Q^2}\}.
 \end{gather*}
Then $S'$ contains one representative 
for each residue class $\beta\pmod \ell$
with 
\begin{gather*}
\pmfrac{1-24\beta}\ell=\pmfrac{1-24\beta_0}\ell=\delta.
\end{gather*}
For each $d$, we see by  \eqref{eq:tddef}     that $t_{d,\beta}$ ranges over those residue classes $t\pmod\ell$ 
with $\pfrac{1-24t}\ell=\delta$ as $\beta$ ranges over
$S'$.
So for each $d$ we have
\begin{gather}\label{eq:fldel}
f_{\ell, \delta}\equiv  \sum_{\beta\in S' }\sum_{n \equiv 24 t_{d,\beta} - 1 \spmod \ell}
  p\pmfrac{n+1}{24} q^{\frac{n}{24}}\pmod\ell.
\end{gather}
Replacing $q$ by $q^\ell$ in \eqref{eq:integralpowers} and using    \eqref{eq:fldel} gives
\begin{gather}\label{eq:concl1}
Q^{-1}  f_{\ell,\delta} \big| V_{Q^2}
+
   \pmfrac{-12 }{Q}  f_{\ell,\delta} \otimes \chi_Q
-
  Q f_{\ell,\delta} \big| U_{Q^2}
\equiv
  0 \pmod{\ell}
\tx{,}
\end{gather}
which is equivalent to \eqref{eq:Qiffone}.
Replacing $q$ by $q^{\ell Q}$ in \eqref{eq:nonintegralpowers}, summing over $\beta\in  S'$, and noting that for fixed~$n$ the Kloosterman sums which appear are independent of the choice of $\beta$,  we obtain
\begin{gather}\label{eq:concl2}
  \sum_{\substack{\pfrac{-Qn}{\ell} = \delta \\ Q\nmid n}}
  K \left(\mfrac{24\beta_0-1}{Q}, \overline{24}^2 \overline{\ell}^4 n,Q \right)
  p\pmfrac{Qn+1}{24}
  q^{\frac{n}{24}}
\equiv
  0 \pmod{\ell}
\tx{.}
\end{gather}

By Lemma~\ref{la:kloosterman_sum_nonzero}, the  Kloosterman sums in~\eqref{eq:concl2} are not divisible by~$\ell$.
It follows that $a_{\ell, \delta}(Q n) \equiv 0 \pmod{\ell}$ for all $n$ with~$Q \nmid n$, which is equivalent to \eqref{eq:Qifftwo}.
This establishes one direction of the third assertion of  Theorem~\ref{thm:UQQVQQ}.

For the other direction, suppose  that \eqref{eq:concl1} and \eqref{eq:concl2} 
hold for some $\beta_0$ with  $ \pfrac{1- 24\beta_0}{\ell} = \delta$ and $Q \mid\mid  24\beta_0-1$. 
By \eqref{eq:fldel}, we see that \eqref{eq:integralpowers} and \eqref{eq:nonintegralpowers} hold with $\beta$ replaced by $\beta_0$.
In other words, we have $H_{\ell, Q^2, \beta_0}\equiv 0\pmod\ell$, from which 
\begin{gather*}
(\ell\tau+Q^2Y)^\frac12g(\ell Q^2, \beta_0, \gamma_{\ell,Q^2} \tau)\equiv 0\pmod\ell.
\end{gather*}
By Lemma~\ref{la:delignerap}
we conclude that $g(\ell Q^2, \beta_0,   \tau)\equiv 0\pmod\ell$.
\end{proof}


\section{Computations}\label{sec:comp}

Throughout this section, we assume that $\ell, Q \ge 5$  are distinct primes.
The source code for each computation is available on the last named  author's homepage (\url{https://www.raum-brothers.eu/martin}).

\subsection{Proof of Theorem~\ref{thm:nosmallcong}}

Suppose that there is a  congruence
\begin{gather*}
  p\big( \ell Q n + \beta_{\ell,Q} \big)
\equiv
  0
  \;\pmod{\ell}
\end{gather*}
as in~\eqref{eq:Qellcong}.
Define $\delta = \left(\frac{1 - 24 \beta_{\ell,Q}}{\ell}\right) \in \{ 0, -1 \}$, and  set
\begin{alignat*}{2}
  \mathcal{B}
&:=
  \Big\{
    \mfrac{24 \beta - 1}{\ell} \,:\,
    \beta \in \Z,\, \left(\mfrac{24 \beta - 1}{\ell}\right) = \delta,\,
    p(\beta) \not\equiv 0 \pmod{\ell}
  \Big\}
\text{,}\quad
&&
  \text{if $\delta = 0$,}
\\
  \mathcal{B}
&:=
  \Big\{
    24 \beta - 1 \,:\,
    \beta \in \Z,\, \left(\mfrac{24 \beta - 1}{\ell}\right) = \delta,\,
    p(\beta) \not\equiv 0 \pmod{\ell}
  \Big\}
\text{,}\quad
&&
  \text{if $\delta = -1$.}
\end{alignat*}
Define~$\ep = \left(\frac{24 \beta_{\ell,Q} - 1}{Q}\right)\in\{\pm1\}$. 
By  Theorem~\ref{thm:raduclasses}, it follows that there are congruences
\begin{gather*}
  p(\ell Q n + \beta) \equiv 0 \pmod{\ell}
\end{gather*}
for all~$\beta \in S_{\ell Q, \beta_{\ell,Q}}$, where~$S_{\ell Q, \beta_{\ell,Q}}$ is as in~\eqref{eq:pmtdef}.
 In particular, we have
\begin{gather}
\label{eq:computations:obstruction}
\left( \mfrac{n}{Q} \right) \ne \ep\ \ \text{for all}\ \ n \in \mathcal{B} 
\text{.}
\end{gather}
Since the set~$\mathcal{B}$ does not depend on~$Q$, this opens the door to rule out possible~$Q$ via a sieve-like computation.

Given a finite subset~$\mathcal{N} \subset \mathcal{B}$ and a bound~$Q_{\mathrm{max}}$, Algorithm~\ref{alg:search} computes the set
\begin{gather}
\label{eq:alg:search:result}
  \mathcal{Q}
:=
  \Big\{
  Q \le Q_{\mathrm{max}} \,:\,
  Q \text{ prime},\,
  \tx{there is }\ep \in \{\pm 1\}\tx{ such that for all }n \in \mathcal{N}, \pmfrac{n}{Q} \ne \ep
  \Big\}
\text{.}
\end{gather}
If~$\mathcal{Q} \subseteq \{2, 3, \ell\}$ for~$\delta = -1$ or 
$\mathcal{Q} \subseteq \{2, 3, 5, 7, 11, \ell\}$ for~$\delta = 0$, we conclude that there are no congruences~\eqref{eq:Qellcong} for~$Q \le Q_{\mathrm{max}}$ that do not  arise from the Ramanujan congruences. If there are  a few excess primes in~$\mathcal{Q}$, we can check each individually to conclude that no corresponding congruence holds.

\SetAlCapSkip{0.5\baselineskip}
\begin{algorithm}[ht]
\caption{Search for congruences~\eqref{eq:Qellcong}}
\label{alg:search}
\KwData{A finite list of positive integers~$\mathcal{N}$, positive integers~$Q_{\mathrm{max}}$ and~$Q_{\mathrm{stride}}$.}
\KwResult{The set~$\mathcal{Q}$ defined in~\eqref{eq:alg:search:result}.}
Choose a subset~$\mathcal{P}$ of primes in $\mathcal{N}$ such that $\prod_{p \in \mathcal{P}} p \approx Q_{\mathrm{max}} \slash Q_{\mathrm{stride}}$\;
$\mathcal{N} \leftarrow \mathcal{N} \setminus \mathcal{P}$\;
$\mathcal{Q} \leftarrow \emptyset$\;
\For{$\ep \in \{ \pm 1 \}$}{
  \For{$p \in \mathcal{P}$}{
    $\mathcal{Q}_p \leftarrow \big\{ n \pmod{p} \,:\, \left(\frac{p}{Q}\right) \ne \ep \big\}$\;
  }
  \For{$(Q_p) \in \prod \mathcal{Q}_p$}{
    \For{$Q \le Q_{\mathrm{max}}$ a probable prime with $Q \equiv Q_p \,\pmod{p}$ for all~$p \in \mathcal{P}$}{
      \If{$\left(\frac{n}{Q}\right) \ne \ep$ for all $n \in \mathcal{N}$}{
        \If{$Q$ is a prime}{
          $\mathcal{Q} \leftarrow \mathcal{Q} \cup \{Q\}$\;
	}
      }
    }
  }
}
\end{algorithm}

We conclude with two remarks on Algorithm~\ref{alg:search}. In practice, the most time consuming part of computing~\eqref{eq:alg:search:result} is to iterate through primes,
 while it is very cheap to evaluate the square-class conditions. For this reason, Algorithm~\ref{alg:search} iterates through probable primes and only filters out 
 non-primes after ensuring the square-class conditions. Splitting a set of primes~$\mathcal{P}$ off the original~$\mathcal{N}$ reduces the number of transitioned
  primes~$Q$ by a factor of~$2^{\#\mathcal{P}}$. The role of~$Q_{\mathrm{stride}}$ in this context is to avoid extra computations that in practice arise from
   imposing the congruences~$Q \equiv Q_p \pmod{p}$. The choice of $Q_{\mathrm{stride}}$ does not impact the result of Algorithm~\ref{alg:search}, but its
    performance. We have observed that changing $Q_{\mathrm{stride}}$ by no more than a factor of~$10$ can deteriorate performance by as much as a factor~$20$.

We provide an implementation that is based on the computer algebra package Nemo \cite{nemo,nemo-2017}. To prove Theorem~\ref{thm:nosmallcong}
we employed it to verify the absence of congruences~\eqref{eq:Qellcong} for $\ell \le 1,000$ and $Q \le 10^{13}$, and $\ell \le 10,000$ and $Q \le 10^{9}$. 
Using a set~$\mathcal{N}$ of size at most~$100$, Algorithm~\ref{alg:search} in the former computation excludes all  but $3019$ pairs~$(\ell,Q)$ 
and in the latter computation all but $2$ pairs.  The exceptional pairs  can all be ruled out by slightly enlarging $\mathcal{N}$.

\subsection{Proof of Corollary~\ref{cor:nocase2}}

In Theorem~\ref{thm:mainpart} and its refinement~\ref{thm:main}, we state that the set~$S$ of primes~$Q$ for which there exists a congruence~\eqref{eq:partcong} for fixed~$\pmfrac{1 - 24 \beta_{\ell,Q}}{\ell} = \delta$ has density zero or the estimate~\eqref{eq:al0} and the Hecke congruences~\eqref{eq:tq0} hold. We have ruled out the latter for all primes~$17 \le \ell \le 10^4$, for $\ell = 11$ if $\delta = -1$, and for $\ell = 13$ if $\delta = 0$ by a computer calculation.  Observe that the missing cases~$\ell \in \{5,7,11\}$ and~$\delta = 0$ correspond to the Ramanujan congruences, and that the cases $\ell \in \{5,7,13\}$ and~$\delta = -1$ correspond to Atkin's congruences for the partition function~\cite[Theorem~1]{atkin_multiplicative}.

Our code relies heavily  on Johansson's performant implementation of the partition function~\cite{johansson} available in ARB~\cite{arb} through Nemo~\cite{nemo,nemo-2017}. The computation proceeds by enumerating primes~$Q \equiv -1 \pmod{\ell}$ until we have found~$n$ with~$(n,Q) = 1$ and~$\pfrac{-n}{\ell} = \delta$ that falsifies the congruence
\eqref{eq:partitionhecke}.

\subsection{Proof of Corollary~\ref{cor:noquadcong}}

We proceed in a similar way to rule out congruences~\eqref{eq:quadcong}. Theorem~\ref{thm:UQQVQQ} lists a number of necessary conditions that can be effectively checked, which we have falsified for all primes~$Q \le 10^4$ and the following primes~$\ell$: for $\ell = \{5,7,11\}$ if $\delta = -1$, for $\ell = 13$ if~$\delta = 0$ and for~$17 \le \ell \le 10^4$ if~$\delta \in \{0,-1\}$.

The implementation is more involved than the previous one, but its effectiveness relies equally on Johansson's implementation of the partition function. Specifically, for each~$\ell$ and~$Q$, we try to falsify~\eqref{eq:Qiffone} by computing
\begin{gather*}
  p\pmfrac{Q^2n+1}{24}
  -
  \pmfrac{-12 \ell}{Q} Q^{-1} 
  p\pmfrac{n+1}{24}
  \pmod{\ell}
\end{gather*}
for a few~$n$ with~$(n,Q) = 1$ and~$\pfrac{-n}{\ell} = \delta$. This suffices to handle many~$Q$. Only if this falsification fails, we employ~\eqref{eq:Qifftwo} by finding~$n$ with~$(n,Q) = 1$ and $\pfrac{-n}{\ell} = \delta$ such that
\begin{gather*}
  p\pmfrac{Qn+1}{24}
\not\equiv
  0
  \pmod{\ell}
\tx{.}
\end{gather*}

\bibliographystyle{amsalpha}
\bibliography{part_cong}

\end{document}